\documentclass[11pt, reqno, a4paper]{amsart}  
\usepackage{amsmath, amsthm, amscd, amsfonts, amssymb, graphicx, xcolor}

\usepackage[all]{xy}
\usepackage{tikz}
\usepackage{here}
\usetikzlibrary{cd}

\usepackage[top=25truemm,bottom=20truemm,left=30truemm,right=30truemm]{geometry}

\newtheorem{thm}{Theorem}[section]
\newtheorem{lemma}[thm]{Lemma}
\newtheorem{proposition}[thm]{Proposition}

\theoremstyle{definition}
\newtheorem{definition}[thm]{Definition}
\newtheorem{example}[thm]{Example}

\newtheorem{remark}[thm]{Remark}

\numberwithin{equation}{section}

\newcommand{\cA}{\mathcal{A}}

\newcommand{\cC}{\mathcal{C}}
\newcommand{\cD}{\mathcal{D}}
\newcommand{\cK}{\mathcal{K}}
\newcommand{\cO}{\mathcal{O}}
\newcommand{\cE}{\mathcal{E}}
\newcommand{\cF}{\mathcal{F}}
\newcommand{\cG}{\mathcal{G}}

\newcommand{\cM}{\mathcal{M}}

\newcommand{\cT}{\mathcal{T}}
\newcommand{\cP}{\mathcal{P}}

\newcommand{\cQ}{\mathcal{Q}}
\newcommand{\cS}{\mathcal{S}}
\newcommand{\cU}{\mathcal{U}}

\newcommand{\cW}{\mathcal{W}}

\newcommand{\hir}{\Sigma_e}

\newcommand{\rdim}{{\rm dim}}
\newcommand{\rgr}{{\rm gr}}

\newcommand{\rid}{{\rm id}}
\newcommand{\rper}{{\rm per}}

\newcommand{\bZ}{\mathbb{Z}}
\newcommand{\bC}{\mathbb{C}}

\newcommand{\bQ}{\mathbb{Q}}
\newcommand{\bR}{\mathbb{R}}
\newcommand{\bP}{\mathbb{P}}

\newcommand{\bH}{\mathbb{H}}
\newcommand{\bfR}{\textbf{R}}
\newcommand{\bfL}{\textbf{L}}

\newcommand{\rhom}{\textrm{Hom}}

\newcommand{\rch}{\textrm{ch}}
\newcommand{\rtd}{\textrm{td}}
\newcommand{\rexp}{\textrm{exp}}

\newcommand{\rdet}{\textrm{det}}
\newcommand{\rdeg}{\textrm{deg}}
\newcommand{\rrank}{\textrm{rank}}
\newcommand{\rStab}{\textrm{Stab}}
\newcommand{\rSpec}{\textrm{Spec}}
\newcommand{\rExt}{\textrm{Ext}}
\newcommand{\rCoh}{{\rm Coh}}

\newcommand{\ramp}{\textrm{Amp}}
\newcommand{\rns}{\textrm{NS}}
\newcommand{\rsup}{\textrm{sup}}

\newcommand{\rRe}{\textrm{Re}}
\newcommand{\rIm}{\textrm{Im}}

\newcommand{\mnH}{\mathnormal{H}}

\usepackage[bookmarksnumbered, colorlinks, plainpages]{hyperref}
\begin{document}
\setcounter{page}{1}

\noindent

\centerline{}
\centerline{}

\title{Wall in the stability space of the gluing stability conditions on Hirzebruch surfaces}
 \author{Yusuke Ohmiya}

\begin{abstract}
This paper investigates the wall structure of the space of stability conditions on Hirzebruch surfaces. Using the gluing construction of \cite{CP} and \cite{Uch} with respect to a fixed semiorthogonal decomposition, we focus on two main objectives: observing the intersection of the geometric chamber $U(S) \subset \rStab(\Sigma_e)$ with the walls of the resulting subspace, and determining the moduli space of $\sigma$-semistable objects on this subspace.
\end{abstract} 
\maketitle

\section{Introduction}
\subsection{Bridgeland's works}
The study of moduli spaces is primarily motivated by classification problems, which are fundamental to many areas of algebraic geometry. To address these problems effectively, the notion of stability is one of the core elements of this theory.

The notion of stability conditions on triangulated categories was introduced by Bridgeland in \cite{Br07}. A stability condition $(Z,\cA)$ consists of a homomorphism $Z : K(\cD) \to \bC$ which respect to the Harder-Narasimhan condition, and an abelian category $\cA$, called a heart of the bounded t-structure. One of the main result of \cite{Br07} is set of stability conditions $\rStab(\cD)$ has the structure of complex manifold. 

It is often described as a generalization of Mumford's slope stability on curves, extending this notion to the derived category of coherent sheaves on a smooth projective variety. Crucially, in the case of surfaces, Bridgeland's stability is known to encompass Gieseker stability. This identification provides a powerful tool, enabling the study of the geometry of Gieseker moduli spaces, by analyzing the wall-crossing phenomena that occur as the stability condition is varied within the manifold $\mathrm{Stab}(S)$ (\cite{Bay09}, \cite{BM14}, \cite{Maci14}). In this identification, the stability condition, known as the divisorial stability condition plays a central role (see Section 4).

\subsection{Geometric chambers and walls}
The geometric chamber, denoted $U(S)$, is the set of stability conditions under which skyscraper sheaves are stable of same phase. For a smooth projective surface defined over $\bC$, it is known to be an open subset of the stability manifold (\cite{MS16}). When an object crosses the wall $\partial U(S) = \overline{U(S)} \setminus U(S)$, it loses stability. Furthermore, it is known that $U(S)$ contains the set of divisorial stability conditions. Therefore, the study of $\partial U(S)$ is essential for the analysis of wall-crossing of moduli spaces.

This approach was investigated for K3 surfaces in \cite{Br08}, and also for birational morphisms $f: X \to Y$ between smooth projective surfaces in \cite{toda2012stabilityconditionsextremalcontractions}, \cite{Tod14}. In \cite{Tod14}, the author shows that there is a correspondence between wall-crossing and the minimal model program. He proves that contractions of curves of self-intersection $-1$ can be realized as wall-crossing in $\rStab(X)$. That is, if $f : X \to Y$ is a birational map contracting a $-1$ curve on $X$, then there is a wall of the geometric chamber such that, after crossing,$M_\sigma([\cO_x]) \cong Y$. In recent years, the general case has been studied in \cite{tramel2018bridgelandstabilityconditionssurfaces}.

\subsection{Gluing construction and perversity}
Research has also proceeded from a different perspective. In \cite{CP}, gluing construction for stability conditions was introduced. The idea is to take a semi-orthogonal decomposition of a triangulated category $\cD = \langle \cD_1,\cD_2 \rangle$, and then combine the stability conditions from each $\cD_i$ to create a new one. Subsequently, \cite{Uch} investigated wall-crossing phenomena on a ruled surface $p : S \to C$ with base curve genus $g(C)>0$ by using this construction, while fixing the Orlov's semiorthogonal decomposition 
$$
    D^b(S) = \langle \cD_1=\bfL p^*D^b(C) \otimes \cO_S(-C_0), \cD_2 = \bfL p^*D^b(C) \rangle.
$$ 
When $\sigma \in \rStab(S)$ is the gluing stability condition of $\sigma_1 \in \rStab(\cD_1)$ and $\sigma_2 \in \rStab(\cD_2)$, the author defined the {\it gluing perversity} $\rper(\sigma)$, as the "phase discrepancy" between $\sigma_1$ and $\sigma_2$ (\cite{Uch} Definition 3.5). Furthermore, it was shown that continuously varying $\rper(\sigma)$ leads to wall-crossings of the $\sigma$-stable objects.

However, this research was restricted to the case where $g(C) > 0$. The $g(C)=0$ case is distinct, because the derived category of coherent sheaves on $C$ possesses a heart of the bounded t-structure $\langle \mathcal{O}(k-1)[j+1],\mathcal{O}(k)[j] \rangle$ arising from Beilinson's exceptional collection (\cite{Bei79}). Consequently, in addition to the standard stability condition \eqref{standardstabilitycondition}, there exists another stability condition derived from this collection (often called a quiver stability condition).

In this paper, we investigate the wall-crossing phenomena induced by the variation of gluing perversity in the $g(C)=0$ case; that is, for a Hirzebruch surface $p : \Sigma_e \to \bP^1$.

The choice between the standard stability condition \eqref{standardstabilitycondition} and this quiver stability condition \eqref{algebraicstabilitycondition} on each component $\mathcal{D}_i$ of the semiorthogonal decomposition naturally leads to the four distinct "glued types" (Case $m=1, 2, 3, 4$) that are investigated in our main result (Definition \ref{typeofgluingstab}, Theorem 1.1).

Furthermore, we investigate the location of this wall (defined by $\rper(\sigma)=0$) within the stability manifold, specifically by calculating its intersection with the boundary of the space of divisorial stability conditions $S_{div} \subset Stab(\Sigma_e)$ (see Section 4).

\subsection{Main Result}
The main results of this paper can be divided into the following two parts.
\subsubsection{Moduli space of $\sigma$-stable objects}
Here we fix $p : \Sigma_e \to \bP^1$ be a Hirzebruch surface defined over $\bC$ with degree $e$, and $D^b(\Sigma_e)$ be its bounded derived category of coherent sheaves. For a gluing stability condition $\sigma = \sigma_{gl} \in \rStab(\Sigma_e)$, we denote $M^{\sigma ss}([\cO_x])$ the set of $\sigma$-semistable objects with the same Chern character as the skyscraper sheaf $\cO_x$. 
\begin{thm}[\bf Theorem \ref{Mainthm1}] \label{mainthm}
    Let $\sigma = \sigma_{gl,m}$ be a gluing stability condition of glued type $m$ (Definition \ref{typeofgluingstab}). Then, we have the following:

    \vspace{3mm}
    {\bf (Case $m=1$)}
    \begin{enumerate}
        \item If $\rper(\sigma)=0$ and $\rper_i(\sigma) \neq 0$ for $i=1,2$, then $\bP^1$ is the coarse moduli space of S-equivalence classes of objects in $M^{\sigma ss}([\cO_x])$.
        \item If $\rper(\sigma)=0$ and $\rper_i(\sigma)= 0$ for $i=1,2$, the moduli space of S-equivalence classes is isomorphic to $\rSpec \ \bC$.
        \item If ${\rm per}(\sigma) \neq 0$ then the moduli space of $\sigma$-semistable objects in $M^{\sigma ss}([\cO_x])$ is empty.
    \end{enumerate}

    \vspace{3mm}
    {\bf (Case $m=2,3)$}
    \begin{enumerate}
        \item If ${\rm per}(\sigma) > 0$, then $\hir$ is the fine moduli space of $\sigma$-stable objects in $M^{\sigma ss}([\cO_x])$.
        \item If $\rper(\sigma)=0$ and $\rper_i(\sigma) \neq 0$ for $i=1,2$, then $\bP^1$ is the coarse moduli space of S-equivalence classes of objects in $M^{\sigma ss}([\cO_x])$.
        \item If $\rper(\sigma)=0$ and $\rper_i(\sigma)= 0$ for $i=1,2$, the moduli space of S-equivalence classes is isomorphic to $\rSpec \ \bC$.
        \item If $\rper(\sigma) < 0$, then the moduli space of $\sigma$-semistable objects in $M^{\sigma ss}([\cO_x])$ is empty.
    \end{enumerate}

    \vspace{3mm}
    {\bf (Case $m=4$)}    
    $\\$ \hspace{5mm} In this case, which necessarily implies $\rper(\sigma)=0$ and $per_i(\sigma)\ne 0$ (see Lemma \ref{phaseoflambda1Oxandrho2Ox}, Lemma \ref{lemmasetofMsigma}), $\bP^1$ is the coarse moduli space of S-equivalence classes of objects in $M^{\sigma ss}([\cO_x])$.
\end{thm}

\subsubsection{Intersection of geometric chamber with wall}
In Section 4, we investigate the position of the wall $\mathcal{W}_{0}$ within the stability manifold by analyzing its relation to the space of divisorial stability conditions $S_{div}$. The subspace $S_{div}$ plays a central role in the theory, as it connects Bridgeland stability with classical notions such as Gieseker stability. The wall $\mathcal{W}_{0}$, defined by the vanishing of the gluing perversity ($\rper(\sigma)=0$), corresponds to the locus where the stability of skyscraper sheaves is critical (see Section 3.3). Determining the intersection of $\mathcal{W}_{0}$ with the boundary of $S_{div}$ provides an explicit geometric description of this wall within the stability manifold. 

Recall that the closure $\overline{S_{div}}$ in $Stab(\Sigma_{e})$ is described as a cone $\{ (x,y,z,w) \in \bR^4 \mid z > 0, w > ze\}$ with two boundary components: $\partial_{z}$ ($z=0$) and $\partial_{w}$ ($w=ze$). \eqref{boundaryofwall}  Our result completely determines the location of $\mathcal{W}_{0}$ relative to these boundaries. 
\begin{thm}[Theorem \ref{thm.boundaryofW01}, Theorem \ref{thm.boundaryofW02}, Theorem \ref{thm.m=3}]
    Let $\mathcal{W}_{0,m}$ be the set of gluing stability conditions of glued type $m$ with zero gluing perversity. (see Definition \ref{localwall}) The intersection $\overline{S_{div}} \cap \mathcal{W}_{0,m}$ is described as follows: 

    \begin{itemize}
        \item (Case $m=1, 2$) The wall $\mathcal{W}_{0,m}$ is entirely contained in the boundary component $\partial_{z}$. It does not intersect the vertex $\partial_z \cap \partial_w$ of $\overline{S_{div}}$.
        \item (Case $m=3$) The wall $\mathcal{W}_{0,3}$ is also contained in the boundary component $\partial_{z}$. In particular, it intersects the vertex of $\overline{S_{div}}$ if and only if the parameters of the associated quiver stability conditions satisfy the determinantal condition:
        $$
            \det\begin{pmatrix} \mathrm{Re}\zeta_{0} & \mathrm{Re}\zeta_{1} \\ \mathrm{Im}\zeta_{0} & \mathrm{Im}\zeta_{1} \end{pmatrix} - \det\begin{pmatrix} \mathrm{Re}(\zeta_{0}+\zeta_{1}) & \mathrm{Re}\zeta_{0}' \\ \mathrm{Im}(\zeta_{0}+\zeta_{1}) & \mathrm{Im}\zeta_{0}' \end{pmatrix} = 0.
        $$
    \end{itemize}
\end{thm}

\vspace{3mm}
\subsection{Outline of this paper}
In Section 2, we review the definition of stability conditions and related foundational results. In Section 3, we construct the gluing stability condition on Hirzebruch surfaces and define the associated gluing perversity. In Section 4, we investigate the intersection of the space of divisorial stability conditions with the wall defined by these gluing stability conditions. Finally, in Section 5, we provide the proof of our main result, Theorem \ref{mainthm}.

\subsection{Notation and Conventions}
Throughout this paper, all varieties are defined over $\mathbb{C}$. For a triangulated (or abelian) category $\cD$ and a subcategory $\cC \subset \cD$, the smallest extension subcategory of $\cD$ containing the objects of $\cC$ is denoted by $\langle \cC \rangle$.
For a point $x \in X$ on a variety $X$, the skyscraper sheaf at $x$ is denoted by $\cO_x$. For a projective line $\mathbb{P}^1$, $\cO$ denotes the structure sheaf of $\mathbb{P}^1$. Similarly, for a surface S, $\cO_{S}$ denotes the structure sheaf of $S$.

\subsection{Acknowledgement}
The author would like to express his gratitude to Takayuki Uchiba for his valuable advice and numerous enlightening discussions. The author also thanks Yuki Matsubara for his extensive discussions and many helpful suggestions. The author is grateful to Tomohiro Karube for his precise remarks regarding the support property. Thanks are also due to Franco Rota for reading the manuscript and providing useful comments. 

\section{Preliminaries}
\subsection{Stability conditions}
Bridgeland introduced the notion of stability condition on triangulated categories in \cite{Br07}, the contents of which we briefly summarize in this section. Throughout the following discussion, we assume that all categories are small. For all triangulated categories and for any objects $A,B \in \cD$, the Hom group $\rhom_{\cD}(A,B)$ is a finite-dimensional vector space over a fixed field $k$.

We follow \cite{Br07} in defining a stability condition using the concept of a heart of a bounded t-structure.

\begin{definition}
    A t-structure of triangulated category $\cD$ is a pair of full subcategories $(\cD^{\leq 0},\cD^{\geq 0})$ that satisfies the following conditions:
    \begin{enumerate}
        \item $\rhom_\cD(\cD^{\leq 0}, \cD^{\geq 0}[-1]) = 0$, 
        \item $\cD^{\leq 0}[1] \subset \cD^{\leq 0}$ and $\cD^{\geq 0}[-1] \subset \cD^{\geq 0}$,
        \item For any object $E \in \cD$, there exists a distinguished triangle
                $$
                    E_{\leq 0} \to E \to E_{\geq 0} \to E_{\leq 0}[1]
                $$
              where $E_{\leq 0} \in \cD^{\leq 0}$ and $E_{\geq 0} \in \cD^{\geq 0}$.
    \end{enumerate}
\end{definition}

A t-structure $(\cD^{\leq 0},\cD^{\geq 0})$ is called {\it bounded} if it satisfies:
$$
    \cD = \bigcup_{i \in \bZ} \cD^{\leq 0}[i] = \bigcup_{i \in \bZ} \cD^{\geq 0}[-i].
$$
\begin{definition}
    Let $(\cD^{\leq 0},\cD^{\geq 0})$ be a bounded t-structure on a triangulated category $\cD$. The {\it heart} of this t-structure is defined as
    $$
        \cA = \cD^{\leq 0} \cap \cD^{\geq 0}.
    $$
\end{definition}

Next, we define a stability function on an abelian category. Let $\cA$ be an abelian category, and let $K(\cA)$ denote its Grothendieck group. Consider a group homomorphism $Z : K(\cA) \to \bC$. The {\it phase} of a non-zero object $E \in \cA$ is defined as 
$$
    \phi(E) = \frac{1}{\pi} arg Z(E) \in (0,1].
$$
An object $E \in \cA$ is called {\it semistable} if for every non-zero proper subobject $0 \neq F \subsetneq E$, the phases satisfy the inequality $\phi(F) \leq \phi(E)$.

\begin{definition}
    A group homomorphism $Z : K(\cA) \to \bC$ is called a {\it stability \ function} on $\cA$ if the following conditions are satisfied.
    \begin{enumerate}
        \item For every non-zero object $E \in \cA$, $Z(E) \in \bH$, where $\bH := \{ r \cdot \rexp(i\pi\phi) \mid r \in \bR_{>0}, \ 0 < \phi \leq 1\}$.
        \item $Z$ satisfies the {\it Harder-Narasimhan property}. Namely, any object $E \in \cA$ has a filtration
                $$
                    0 = E_0 \subset E_1 \subset \cdots \subset E_n = E
                $$
                such that the factors $F_i = E_i / E_{i-1}$ are the semistable objects in $\cA$ with
                $$
                    \phi(F_1) > \cdots > \phi(F_n).
                $$
    \end{enumerate}
\end{definition}

\begin{definition}
     \begin{enumerate}
         \item A pair $(Z,\cA)$ is called a {\it pre-stability condition} if $\cA$ is the heart of a bounded t-structure and $Z : K(\cA) \to \bC$ is a stability function on it.
         \item Let $\cD$ be a triangulated category with given a homomorphism $v : K(\cD) \cong K(\cA) \rightarrow \Lambda$, where $\Lambda$ is a free Abelian group of finite rank. A pre-stability condition $(Z,\cA)$ called {\it with respect to} $\Lambda$ if $Z$ factor through a homomorphism $v$.
     \end{enumerate}
\end{definition}

\begin{definition} \label{dfn.stabilitycondition}
     Let $\sigma=(Z,\cA)$ is a pre-stability condition with respect to a lattice $\Lambda$. Then $\sigma$ called a {\it stability condition} if it satisfies the following condition.
    \begin{itemize}
        \item (Support property) : There exists a $C \in \bR_{>0}$ such that
        \begin{equation} \label{supportproperty}     
            {\rm inf} \left \{ \frac{|Z(E)|}{\| v(E) \|} \ \middle  | \ E \neq 0\in \cP(\phi), \phi \in \bR \right \} \geq C
        \end{equation}
        for some (and equivalently, for all) choices of norm $\| \cdot \|$ on $\Lambda_\bR = \Lambda \otimes_\bZ \bR$.
    \end{itemize}
\end{definition}

We write $\rStab_\Lambda(\cD)$ for the set of stability conditions on $\cD$ with respect to a fixed lattice $\Lambda$. This definition is equivalent to defining a stability condition as a pair $\sigma = (Z,\cP)$, where $\cP$ is a {\it slicing} of $\cD$, as defined below.
\begin{definition}
    A {\it slicing} on a triangulated category $\cD$ is a collection of full additive subcategories $\cP(\phi) \subset \cD$ for all $\phi \in \bR$ that satisfy:
    \begin{enumerate}
        \item $\cP(\phi)[1] = \cP(\phi+1)$,
        \item $\rhom(A,B)=0$ if $A \in \cP(\phi_1)$ and $B \in \cP(\phi_2)$ with $\phi_1 > \phi_2$,
        \item For any object $E \in \cD$ there exists a decomposition
                $$
                  \xymatrix@=15pt{
                        0 = E_0 \ar[rr] & &E_1 \ar[ld] \ar[rr] & & E_2 \ar[ld] \ar[r] & \cdots \ar[r] & E_{l-1} \ar[rr] & & E_l = E \ar[ld] \\
                          & A_1 \ar[lu]& & A_2 \ar[lu] & & & & A_l \ar[lu] & & &
                    }
                $$
               such that $A_i \in \cP(\phi_i)$ and $\phi_i > \phi_{i+1}$ for any $i$.
    \end{enumerate}
\end{definition}

Given a pre-stability condition $(Z,\cA)$, the corresponding slicing $\cP$ is constructed as follows: For each $\phi \in \bR$, 
$$
    \cP(\phi) = \{ E \in \cA \mid E \textrm{ is semistable of phase } \phi \}.
$$
An object in $E \in \cP(\phi)$ is said to be {\it semistable of phase} $\phi$. Such an object is called {\it stable} if it is simple, meaning it has no proper subobjects in $\cP(\phi)$. When we need to be precise, we refer to these as {\it $\sigma$-semistable} or {\it $\sigma$-stable} to specify the dependence on the stability condition $\sigma = (Z,\cP)$.

For a smooth projective variety $S$, we write $K_{num}(S)$ for the numerical Grothendieck group of $S$, defined as $K(S) / {\rm Ker} \  \chi(-,-)$, where $\chi(-,-)$ is a Euler form on $D^b(S)$. We have that $K_{num}(S)$ is a finite generated $\bZ$-lattice. Furthermore, when $S$ is surface, we have
$$
    K_{num}(S) \cong \mnH^*_{alg}(S) := \mnH^0(S,\bZ) \oplus \rns(S) \oplus \mnH^4(S,\bZ).
$$
The choice of $\Lambda = K_{num}(S)$ provied an example of surjective group homomorphism $v : K(D^b(S)) \to \Lambda$ from \eqref{dfn.stabilitycondition}. In our assumption, $v$ coincides the Chern character map 
$$
    \rch : K(D^b(S)) \to \mnH^*(S,\bQ)
$$
and $\Lambda = \rIm(\rch)$. We say that a stability condition $\sigma = (Z,\cP) \in \rStab(S)$ is {\it numerical} if $Z$ factor through $K_{num}(S)$. In this paper, unless otherwise specified, $\Lambda = K_{num}(S)$ is fixed. 

We now briefly describe the deformation theory of $\rStab_\Lambda(S)$ intoroduced in \cite{Br07}. Let the symbol $\| \cdot \|$ denotes the norm on $\rhom_{\bZ}(\Lambda,\bC)$ induced by $\Lambda_\bR = K_{num}(S) \otimes_\bZ \bR$. Recall that any choice of norm is equivalent, and subsequent definitions do not depend on it. For any two stability conditions $\sigma_1,\sigma_2 \in \rStab_\Lambda(S)$, the generalized metric defined as
$$
    d(\sigma_1,\sigma_2) = \rsup \{ |\phi_{\sigma_1}^+(E) - \phi^+_{\sigma_2}(E)|,|\phi_{\sigma_1}^-(E) - \phi^-_{\sigma_2}(E)|, \| Z_1 - Z_2 \| \mid 0 \neq E \in D^b(S)\}
$$
on the space of slices on $D^b(S)$, where $\phi^+\sigma(E)$ and $\phi^-\sigma(E)$ are the largest and smallest phase of a Harder-Narasimhan factor of $E$, respectively. The topology on $\rStab_\Lambda(S)$ has a basis of open sets defined for any $\sigma_0 = (Z_0,\cP_0)$ by
$$
    B_\epsilon(\sigma_0) := \{ \sigma =(Z,\cP) \mid \| Z-Z_0 \| < {\rm sin} (\pi \epsilon), d(\cP_0,\cP) < \epsilon \}.
$$

\begin{proposition}[\cite{Br07} Theorem 7.1] \label{thm.Br07Theorem7.1}
    Let $\sigma_0 = (Z_0,\cP_0)$ be a locally finite stability condition on a triangulated category $\cD$. Then there is an $\epsilon_0 > 0$ such that if $\epsilon_0 > \epsilon > 0$ and $Z : K(S) \to \bC$ is a group homomorphism satisfying 
    $$
        |Z(E)-Z_0(E)| < {\rm sin}(\pi \epsilon) |Z_0(E)|
    $$
    for all $E \in \cD$ semistable in $\sigma_0$, then there is a locally finite stability condition $\sigma=(Z,\cP)$ on $\cD$ with $d(\cP_0,\cP) < \epsilon$.
\end{proposition}

Let $\sigma = (Z,\cP)$ be a stability condition in $\rStab_\Lambda(S)$. A subcategory $\cQ \subset D^b(S)$ called the {\it thin subcategory} if $\cQ$ is formed $\cQ((a,b))$ with $0 \leq b-a \leq 1-2 \epsilon$ for some $\epsilon > 0$. A $\sigma$-semistable object $E \in D^b(S)$ called {\it enveloped} by $\cQ((a,b))$ if the phase of it satisfies $a+\epsilon \leq \phi(E) \leq b+\epsilon$. 

The slices $\cP$ in the Proposition \ref{thm.Br07Theorem7.1} are constructed as a family of following full additive subcategories: For all $\psi \in \bR$, 
\begin{equation}
    \cP(\psi) = \left \{ F \in D^b(S) \middle | \
    \begin{aligned}
        & Z\textrm{-semistable of phase } \psi, \textrm{which included by} \\ 
        & \textrm{some enveloping thin subcategory \ } \cP_0((a,b)) \\ 
    \end{aligned}
    \right \} \cup \{0\}. \label{eq.envelopedsubcat}
\end{equation}
\vspace{3mm}

\subsection{Gluing construction} 
In this section, we review the gluing construction of stability conditions in \cite{CP}. Throughout this section, we assume that $\cD$ is a triangulated category equipped with a semi-orthogonal decomposition $\cD = \langle \cD_1,\cD_2 \rangle$. By definition, this means that $\cD_1$ and $\cD_2$ are triangulated subcategories in $\cD$ such that $\rhom(E_2,E_1) = 0$ for every $E_1 \in \cD_1$ and $E_2 \in \cD_2$. Furthermore, for every $E \in D$, there exists an exact triangle:
\begin{equation} \label{gluingexactseq}
    E_2 \to E \to E_1 \to E_2[1]
\end{equation}
where $E_i \in \cD_i$ for $i=1,2$. The object $E_1$ is given by $\lambda_1(E)$, where $\lambda_1$ is the left adjoint to the inclusion $\cD_1 \hookrightarrow \cD$, and $E_2$ given by $\rho_2(E)$ where $\rho_2$ is the right adjoint to the inclusion $\cD_2 \hookrightarrow \cD$ respectively.

\begin{definition}[\cite{CP} \S 2] \label{gluingstabilitycondition}
     Let $\sigma_i = (Z_i,\cA_i)$ be a pre-stability condition on $\cD_i$ for $i=1,2$. A pre-stability condition $\sigma=(Z,\cA)$ on $\cD$ is called a \textit{gluing pre-stability condition} of $\sigma_1$ and $\sigma_2$ if it satisfies the following conditions: 
    \begin{enumerate}
        \item $Z = Z_1 \circ \lambda_1 + Z_2 \circ \rho_2$, \label{glfunc} 
        \item $\cA = \{ E \in \cD \ | \ \lambda_1(E) \in \cA_1 \ and \ \rho_2(E) \in \cA_2 \}$. \label{glheart}
    \end{enumerate}
    A gluing pre-stability condition $\sigma$ is called a \textit{gluing stability condition} if it also satisfies the following condition for $E_1 \in \cA_1$ and $E_2 \in \cA_2$:
    \begin{equation}
        \rhom_{\cD}^{\leq 0}(E_1,E_2) = 0. \label{gleq1}
    \end{equation}

\end{definition}

The condition \eqref{gleq1} is known as the \textit{gluing condition}.
Here, we write $\cG l(\sigma_1,\sigma_2)$ (resp.  $\cG l_{pre}(\sigma_1,\sigma_2)$) as the set of gluing stability conditions (resp. gluing pre-stability conditions) of $\sigma_1 \in \rStab(\cD_1)$ and $\sigma_2 \in \rStab(\cD_2)$. 
\vspace{3mm}

\subsection{Inducing stability conditions}
 \label{inducingstabilitycondition}
In this section, we review the inducing stability conditions as described in \cite{Mac09}.
Let $F : \cD \to \cD'$ be an exact functor between two essentially small triangulated categories. 
Assume that $F$ satisfies the following condition:
\begin{equation}
    \rhom_{\cD'}(F(A),F(B)) = 0 \implies \rhom_{\cD}(A,B) = 0. \label{Indeq1}
\end{equation}

\begin{definition}[\cite{Mac09} \S 2.2] \label{indstab}
    Let $\sigma' = (Z',\cP')$ be a stability condition on $\cD'$. Then the \textit{inducing stability condition} $\sigma = F^{-1}\sigma := (Z,\cP)$ defined by
    $$
        Z = Z' \circ F_*, \ \ and \ \  \cP(\phi) = \{ A \in \cD \ | \ F(A) \in \cP'(\phi) \},
    $$
    where $F_* : K(\cD) \otimes \bC \to K(\cD') \otimes \bC$ is the natural morphism on Grothendieck groups induced by $F$.
\end{definition}

\begin{definition}
    Let $F : \cD \to \cD'$ be an exact functor. An abelian category $\cA \subset \cD$ is called $F$-admissible if 
    \begin{enumerate}
        \item $\cA$ is the heart of bounded t-structure, 
        \item $\rhom_{\cD'}^{< 0}(F(A),F(B)) = 0$, for all $A,B \in \cA$,
        \item The restriction of $F$ to $\cA$ is full.
    \end{enumerate}
\end{definition}

\begin{proposition}[\cite{Mac09} Proposition 2.12] \label{Mac09Prop2.12}
    Let $F : \cD \to \cD'$ be an exact functor which satisfies the condition \eqref{Indeq1} and assume that $\langle F(\cD) \rangle = \cD'$. Let $\sigma = (Z,\cP) \in \rStab(\cD)$ be such that its heart $\cP((0,1])$ is of finite length with a finite number of minimal objects. Assume furthermore $\cP((0,1])$ is $F$-admissible. Then $F_* : K(\cD) \to K(\cD')$ is an isomorphism. Define $\sigma' = F(\sigma) = (Z',\cP')$ where $Z' = Z \circ F_*^{-1}$ and $\cP'((0,1]) = \langle F(\cP((0,1])) \rangle$. Then $\sigma'$ is a locally finite stability condition on $\cD'$. Moreover $F^{-1}(\sigma') = \sigma$.
\end{proposition}

\subsection{Exceptional collection on triangulated categories} \label{exceptionalcollection}
In this section, we recall the concept of an exceptional collection in a triangulated category following \cite{HuyFM}, \cite{Mac07} and \cite{maiorana2018modulisemistablesheavesquiver}. We assume that $\cD$ be a $\bC$-linear triangulated category. This means that for any two objects $E,F \in \cD$ the $\bC$-vector space $\bigoplus_i \rhom_{\cD}^i(E,F)$ is finite dimensional. 

\begin{definition}
    \begin{enumerate}
        \item  An object $E \in \cD$ is called {\it exceptional} if it satisfies
                $$
                    \rhom^i(E,E) = 0 \ \text{for} \ i \neq 0, \quad \textrm{and} \quad \rhom^0(E,E) = \bC.
                $$
        \item  An orderd collection of exceptional objects $\{E_1,...,E_n\}$ is called {\it exceptional collection} if it satisfies
                $$
                    \rhom(E_i,E_j) = 0 \ \textrm{for all} \ i > j.
                $$
                An exceptional collection is called {\it strong} if $\rhom^k(E_i,E_j) = 0$ for all $i,j$ and $k \neq 0$.
        \item  An exceptional collection $\cE = \{E_1,...,E_n\}$ is said to be {\it Ext} if 
                $$
                    \rhom_{\cD}^{\leq 0}(E_i,E_j) \ \textrm{for all} \ i \neq j.
                $$
                It is called {\it full} if it generates $\cD$ through shifts and extensions.
    \end{enumerate}
\end{definition}

An Ext-exceptional collection can be constructed from a strong one. For instance, if $\{E_1,...,E_n \}$ is a strong, full exceptional collection on $\cD$, then $\langle E_1[n],E_1[n-1],...,E_n \rangle$ is an Ext-exceptional collection.

\begin{example}
    Let $X = \bP^n$. The derived category $D^b(\bP^n)$ has a strong, full exceptional collection known as Beillinson's exceptional collection: $\{ \cO,\cO(1),...,\cO(n) \}$. Correspondingly, $\{ \cO[n],\cO(1)[n-1],...,\cO(n) \}$ is an Ext-exceptional collection on $D^b(\bP^1)$.
\end{example}

Finally, the {\it left dual} $^\vee\cE = (^\vee E_n, ..., ^\vee E_1)$ and the {\it right dual} $\cE ^\vee = (E_n^\vee , ..., E_1^\vee )$ of an exceptional collection $\{ E_1,...,E_n \}$ are defined by the conditions: 
$$
    \rhom_\cD(^\vee E_i,E_j[l]) = \left \{\begin{aligned}  \bC \quad  &{\rm if \ } i=j=n-l \\ 0 \quad  & {\rm otherwise} \quad , \end{aligned} \right .
$$
$$
    \rhom_\cD( E_i,E_j^\vee [l]) = \left \{\begin{aligned}  \bC \quad  &{\rm if \ } i=j=l \\ 0 \quad  & {\rm otherwise} \quad .\end{aligned} \right .
$$
Given a full exceptional collection, its left and right dual always exist, are unique, and are also full. These can be constructed by repeated mutations \cite{maiorana2018modulisemistablesheavesquiver}. It is also worth noting that if a full exceptional collection exists, the Euler form $\chi$ is nondegenerate, and $K(\cD) \cong  K_{num}(\cD)$ is freely generated by the elements in the collection. 

Let $\mathcal{E} = (E_1, \dots, E_n)$ be a full strong exceptional collection of $\mathcal{D}$ and set $T = \bigoplus_{i=1}^n E_i$. It is a well-known fact that the endomorphism algebra $A := \mathrm{End}(T)$ is a basic algebra. Therefore, $A$ is isomorphic to the path algebra of a quiver $Q$ with relations $I$:
$$
    A \cong \bC Q / I.
$$
The vertices of $Q$ are labeled by $1, \dots, n$. Since $\mathcal{E}$ is an exceptional collection (specifically $\mathrm{Hom}(E_j, E_i) = 0$ for $j > i$), $Q$ is an ordered quiver, meaning there are no arrows from vertex $j$ to vertex $i$ if $j > i$. Furthermore, let $D^b(X)$ be the bounded derived category of coherent sheaves on a smooth projective variety $X$, and let $\mathbf{Mod}^{\mathrm{f.g.}}_A$ be the category of finitely generated right $A$-modules. Then,
\begin{thm}[\bf \cite{Bon98} Theorem 6.2] \label{Bondalequiv}
    With the notation above, we have a triangulated equivalence:
    $$
        \Phi : \bfR \rhom(T,-) : D^b(X) \to D^b(\mathrm{\mathbf{Mod}}^{\mathrm{f.g.}}_{A}).
    $$
\end{thm}

\vspace{3mm}
\begin{example}
    In the case of $X = \bP^n$, the corresponding algebra $\bC Q / I$ is given by the following quiver with relations:
    $Q$ has $n+1$ vertices, $\{ \cO,....,\cO(n) \}$, and $n(n+1)$ arrows $\phi_{j,k} : \cO(k) \to \cO(k+1)$, which correspond to a basis of the global sections $\mnH^0(\bP^n,\cO(1))$. These arrows are indexed by $j=0,...,n$ and $k=0,...,n-1$. The relations are given by $\phi_{j,k} \phi_{j',k+1} = \phi_{j',k} \phi_{j,k+1}$.
\end{example}

\vspace{3mm}
Based on the following theorem, there is a method for constructing a heart of bounded t-structure from an Ext-exceptional collections.
\begin{thm}[\cite{Mac07} Lemma 3.14 and Lemma 3.16] \label{Mac07Lem3.143.16}
    Let $\langle E_1,...,E_n \rangle$ be a full Ext-exceptional collection in $\cD$ then the category generated by extensions $\langle E_1,...,E_n \rangle$ is a heart of a bounded t-structure. Assume that $(Z,\cP)$ is a stability condition and $E_1,...,E_n$ are all in $\cP((0,1])$ for some $\phi \in \bR$, then $\langle E_1,...,E_n \rangle = \cP((0,1])$ and each $E_i$ is stable.
\end{thm}

Let $\cA = \langle E_1,...,E_n \rangle$ be a heart generated by an full Ext-exceptional collection on $\cD$. The following is a method for construct a stability condition from $\cA$, which introduced in \cite{Mac07}: Fixed $\zeta_1,...,\zeta_n \in \bH$, and define a homomorphism $Z : K(\cA) \to \bC$ as 
$$
    Z(E_i) = \zeta_i \quad {\rm for \ }i = 1,...,n.
$$
By the Theorem \ref{Mac07Lem3.143.16}, a pair $(Z,\cA)$ admits a stability condition on $\cD$ where $\cP((0,1]) = \cA$. Such stability conditions are called {\it quiver} or {\it algebraic}.

\subsection{Stability conditions on $D^b(\bP^1)$}
The structure of the space of stability conditions on a smooth projective curve $C$ depends on its genus. According to \cite{Oka05}, when $g(C)=0$ (i.e. $C \cong \bP^1$), the space $\rStab(\bP^1)$ is homeomorphic to $\bC^2$. This contrasts with the case where $g(C) > 0$, for which $\rStab(C)$ is homeomorphic to $\widetilde{GL^+(2,\bR)}$. (\cite{Mac07} Theorem 2.7). This structural difference arises because, when $g(C)=0$, the derived category $D^b(\bP^1)$ contains the exceptional collections.

Recall that $D^b(\bP^1)$ has a strong, full exceptional collection $\{ \cO(k-1),\cO(k) \}$ for all $k \in \bZ$. From this, one can construct the Ext-exceptional collection $\{ \cO(k-1)[1], \cO(k) \}$. By Theorem \ref{Mac07Lem3.143.16}, the extension closure of this correction
$$
    \langle \cO(k-1)[1], \cO(k) \rangle,
$$
forms the heart of a bounded t-structure on $D^b(\bP^1)$. (\S \ref{exceptionalcollection})

The following lemma will be used in subsequent sections.
\begin{lemma}[\cite{Oka05} Lemma 3.11] \label{lemOka3.1}
    \begin{enumerate}
        \item For $n \in \bZ$ there exists exact triangle on $D^b(\bP^1)$
            \begin{enumerate}
                \item $\cO(k+1)^{n-k} \rightarrow \cO(n) \rightarrow \cO(k)^{n-k-1}[1],  (if \ n>k+1) $
                \item \hspace{2mm} $\cO(k+1)^{k-n}[-1] \rightarrow \cO(n) \rightarrow \cO(k)^{k-n+1}$,  (if $n<k$).
            \end{enumerate}
        \item \hspace{2mm} For $k \in \bZ$ and $x \in \bP^1$ there exsists exact triangle 
                $$
                    \cO(k) \rightarrow \cO_x \rightarrow \cO(k-1)[1].
                $$
    \end{enumerate}
    where $\cO_x$ is the skysclaper sheaf of $\bP^1$. 
\end{lemma}

\begin{thm}[\cite{Oka05} Corollary 3.4] \label{Oka05.cor3.4}
    The hearts of bounded t-structure on $D^b(\bP^1)$ that admits a central charge is one of the following:
\begin{enumerate}
    \item $\rCoh(\bP^1)[j]$
    \item $\cA_{k,j,p} = \langle \cO(k-1)[j+p], \cO(k)[j] \rangle$
\end{enumerate}
for $k,j \in \bZ$ and $p > 0$. 
\end{thm}

\subsection{Gluing construction on Ruled surfaces}
In this section, we briefly review the relevant content from \cite{Uch}. A {\it ruled surface} is a smooth projective surface $S$ equipped with a surjective morphism $p : S \to C$ to a smooth projective curve, such that every fiber $f$ over a point $x \in C$ is isomorphic to $\bP^1$. The projection $p$ admits a section $s : C \to S$ satisfying $s \circ p =\rid$. The image $s(C)$ denoted by $C_0$. The sheaf $\cE = \bfR p_*\cO_S(C_0)$ is a rank $2$ vector bundle on $C$, and $S$ is isomorphic to the projective bundle $p : \bP(\cE) \to C$.

In the particular case where $g(C)=0$, the surface $S$ is called a {\it Hirzebruch surface}. It is a projective bundle over $\bP^1$ associated with a locally free sheaf $\cE = \cO \oplus \cO(-e)$ for some integer $e \geq 0$. The number $e$ referred to as the {\it degree} of $S$. The intersection numbers on $S$ are given by $C_0^2 = -e$, $C_0.f = 1$ ,$f^2 = 0$. The canonical divisor of $S$ is given by
$$
    K_S = -2C_0 + (2g-2 + \rdeg(\cE))f.
$$
The Neron-Severi group $\rns(S)$ is generated by $C_0$ and $f$. For $E \in D^b(S)$, we will use the short-hand notation $E(nC_0+mf)$ for $E \otimes^{\bf L} \cO_S(nC_0+mf)$. 

We now review how to construct gluing stability conditions on $S$, following \cite{Uch}. Since the projection $p$ is a flat, the left derived pullback $\bfL p^*$ is an exact functor. The derived tensor product $\otimes^{\bfL}$ is also exact since $S$ is smooth. For simplisity, we will use the symbols $p^*$ and $\otimes$ instead of $\bfL p^*$ and $\otimes^{\bfL}$, respectively. Applying Orlov's projective bundle formula(\cite{Orl}), we obtain the following semi-orthogonal decomposition:
$$
    D^b(S) = \langle  p^*D^b(C) \otimes \cO_S(-C_0) ,  p^*D^b(C) \rangle.
$$

Let $\cD_1 = p^*D^b(C) \otimes \cO_S(-C_0)$ and $\cD_2 = p^*D^b(C)$. As a consequence of the section $\S 2.3$, the spaces of stability conditions on these categories can be described as
\begin{equation}    
    \rStab(\cD_1) = \left \{ (Z_1,\cP_1) \ \middle  | \   
        \begin{split}
            & (Z,\cP) \in \rStab(C), Z_1 = Z \circ F_1^{-1}, \\ 
            & \cP_2(\phi) = p^*\cP(\phi)    \otimes \cO_S(-C_0) \ {\rm for \ all} \ \phi \in \bR  
        \end{split}
        \right \},  \label{IndStabD1}
\end{equation}
\begin{equation}
    \rStab(\cD_2) = \left \{ (Z_2,\cP_2) \ \middle | \  
        \begin{split}    
            &(Z,\cP) \in \rStab(C), Z_2 = Z \circ F_2^{-1}, \\ 
            & \cP_2(\phi) = p^*\cP(\phi) \ {\rm for \ all} \ \phi \in \bR. 
        \end{split} \right \},  \label{IndStabD2}
\end{equation}
where 
\begin{equation}
    F_1 : K(D^b(C)) \cong K(\cD_1), \quad F_2 : K(D^b(C)) \cong K(\cD_2)  \label{isomofGrothendieckgroups}
\end{equation}
are the canonical isomorphisms of Grothendieck groups(\cite{Uch} \S 3).

Note that $\cD_1$ has a left adjoint functor $\lambda_1 : D^b(S) \to \cD_1$ of embedding $\cD_1 \hookrightarrow D^b(S)$. Similary, $\cD_2$ has a right adjoint functor $\rho_2 : D^b(S) \to \cD_2$. In our cases, by using the Grothendieck-Verdier duality, we can explicitly calculate these functors.
\begin{lemma}[\cite{Uch} Proposition 3.3] \label{Uch.Proposition3.3}
    Let $E$ be an object of $D^b(S)$. Then
    \begin{enumerate}
        \item $\lambda_1(E) = p^*(\bfR p_* E(-C_0 + (2g(C)-2 + \rdeg(\cE))f) \otimes T_C [1] ) \otimes \cO_S(-C_0)$, 
        \item $\rho_2(E) = p^* \bfR p_* E$
    \end{enumerate}
    where $T_C$ is a tangent bundle on $C$.
\end{lemma}

One of the main porposes of \cite{Uch} is describe the boundary of the geometric chamber of divisorial stability conditions, which we denote by $\partial \overline{S_{div}}$. This boundary is characterized by the concept of {\it gluing perversity}. 
\begin{definition}[\cite{Uch} Definition 3.6]
    Let $\sigma_{st} = (Z_{st},\cP_{st})$ be the standard stability condition on the base curve. Suppose that $\sigma_1 = (Z_1,\cP_1) \in \rStab(p^*D^b(C) \otimes \cO_S(-C_0))$ with $\cP_1(0) = p^*\cP_{st}(\phi_1) \otimes\cO_S(-C_0)$ and $\sigma_2 = (Z_2,\cP_2) \in \rStab(p^*D^b(C))$ with $\cP_2(0)=p^*\cP_{st}(\phi_2)$. Assume that $\sigma$ is a gluing pre-stability condition of $\sigma_1$ and $\sigma_2$. Then the {\it gluing perversity} defined to be $\rper(\sigma) = \phi_1-\phi_2$.
\end{definition}

The author proves that $\partial \overline{S_{div}}$ is precisely the set of gluing stability conditions with gluing perversity 1 (\cite{Uch} Lemma 3.10, Lemma 4.2, and Lemma 4.3). The author also described a matrix that transforms from gluing one to the boundary in the following theorem. A set $S_{gl,1}$ denotes the gluing stability conditions which has gluing perversity $1$.

\begin{thm}[\cite{Uch} Theorem 4.4]
    Let $S_{div}$ be the set of divisorial stability conditions on $S$. Suppose that $A=\begin{pmatrix}
        a & -\frac{1}{2}ea \\
        0 & a
    \end{pmatrix} \in \widetilde{GL^+(2,\bR)}$ with $a < 0$ and $-e = \rdeg \ \cE$. Then $\partial \overline{S_{div}} \cap S_{gl,1}$ is the set of $\widetilde{GL^+(2,\bR)}$-translates of a stability condition glued from $A.\sigma_1$ and $\sigma_2$.
\end{thm}

However, as previously explained, when the genus of base curve is zero, the stability space $\rStab(C)$ has a different structure from the positive genus case due to the existence of Beillinson's exceptional collection. To resolve this problem, we will use a slightly different method in \S 3 and \S 4.

\section{Construction}
In this section, we construct the gluing stability condition where the base curve has genus zero. Let $p : \Sigma_e \to \bP^1$ denote a Hirzebruch surface.  We keep the notation in the previous section and fix a semiorthogonal decomposition
\begin{equation}
    D^b(\hir) = \langle p^*D^b(\bP^1) \otimes \cO_{\hir}(-C_0), p^*D^b(\bP^1) \rangle. \label{SOD1} \\
\end{equation}

\subsection{Gluing construction on $D^b(\Sigma_e)$}
First, we examine the gluing construction of pullbacks of hearts on $D^b(\bP^1)$. According to \cite{Oka05},  there are two types of hearts of bounded t-structures that admit a stability function:
\begin{itemize}
    \item $\rCoh(\bP^1)[j]$, 
    \item $\langle \cO(k-1)[j+p],\cO(k)[j]\rangle$ 
\end{itemize}
for $k,j \in \bZ$ and $p > 0$.We will focus on the case where $p=1$, and denote the heart $\langle \cO(k-1)[j+1],\cO(k)[j]\rangle$ by $\cA(k)[j]$. The functor $p^*$ and $- \otimes \cO_{\hir}(-C_0)$ preserve the structure of the heart of a bounded t-structure (\cite{Uch} \S 3). Thus, the glued heart can be formed in one of the following four ways:
\begin{enumerate} \label{gluableheart}
    \item $\cA_1[j_1] = p^*\rCoh(\bP^1)[j_1] \otimes \cO_{\hir}(-C_0)$ and $\cA_2[j_2] = p^*\cA(k)[j_2]$
    \item $\cA_1[j_1] = p^*\cA(k)[j_1] \otimes \cO_{\hir}(-C_0)$ and $\cA_2[j_2] = p^*\rCoh(\bP^1)[j_2]$
    \item $\cA_1[j_1] = p^*\cA(k)[j_1] \otimes \cO_{\hir}(-C_0)$ and $\cA_2[j_2] = p^*\cA(k')[j_2]$
    \item $\cA_1[j_1] = p^*\rCoh(\bP^1)[j_1] \otimes \cO_{\hir}(-C_0)$ and $\cA_2[j_2] = p^*\rCoh(\bP^1)[j_2]$.
\end{enumerate}
We determine the conditions under which the gluing condition \eqref{gleq1} is satisfied.
\begin{proposition} \label{Gluingheart}
    In all cases of above, $\cA_1[j_1]$ and $\cA_2[j_2]$ satisfiy the gluing condition \eqref{gleq1} if and only if $j_1 > j_2$.
    \begin{proof}
        Since $p^*$ is fully faithful, the projection formula yields the following isomorphism for $E_1 \in \cA_1$ and $E_2 \in \cA_2$:
        \begin{equation} \label{projectionformula1}            
            \rhom_{D^b(\hir)}(p^*E_1 \otimes\cO_{\hir} (-C_0), p^* E_2))  \cong \rhom_{D^b(\bP^1)}(E_1,E_2 \otimes (\cO\oplus \cO(-e))).
        \end{equation}
        Therefore, the problem reduces to checking the vanishing of $\rhom^{i\leq 0}_{D^b(\bP^1)}$ of the corresponding objects of $D^b(\bP^1)$. Any object $E \in \cA(k)[j]$ is quasi-isomorphic to a complex of vector bundles
        $$
            \cdots \to 0 \to \cO(k-1)^{\oplus n_{j-1}} \to \cO(k)^{\oplus n_{j}} \to 0 \to \cdots
        $$
        which concentrated in degrees $j-1$ and $j$ in $D^b(\bP^1)$. On the otherhands, any object of $\rCoh(\bP^1)[j']$ is quasi-isomorphic to a complex
        \begin{equation}
           E^\bullet:= \cdots \to 0 \to \cO(m_{j'-1})^{\oplus n_{j'-1}} \to \cO(m_{j'})^{\oplus n_{j'}} \to 0 \to \cdots
        \end{equation}
        concentrated in $j'-1$ and $j'$ by the Grothendieck-Birkoff's theorem. 
        
        For the nessesity of (1), we assume that $j_1 > j_2$. By the definition of derived hom, it is sufficient to show that the condition
        \begin{equation}            
            \bigoplus_{i \leq 0} \rExt^i (\cO(k-1)[j_1+1],E^\bullet[j_2]) = \bigoplus_{i \leq 0} \rExt^i (\cO(k)[j_1],E^\bullet[j_2]) = 0. \label{eq.extofP1}
        \end{equation}
        Notice that in general, $\rExt^i(A[j],B[j']) = \rExt^{i+j'-j}(A,B)$ holds. For the right hand side of \eqref{eq.extofP1}, the condition $i \leq 0$ combined with our assumption $j_1>j_2$ ensures that the degree $i+j_2-j_1$ is always negative. Since the homological dimension of $D^b(\bP^1)$ is $1$, any Ext-group with a degree other than $0$ or $1$ must be zero. Therefore, the required Ext-groups vanish. For the left hand side of \eqref{eq.extofP1}, the relevant Ext-group has degree $(i+j_2-j_1)-1$. This degree is therefore also always negative under the same assumption. Thus the relevant Ext-groups also vanish and proving the necessity of (1). By changing the assumptions of $\cA_1$ and $\cA_2$, and performing the same argument, we obtain the necessity of (2), (3), (4). 
        
        In order to prove of sufficiently condition, we must show that the aforementioned Ext-groups \eqref{eq.extofP1} do not vanish when $j_2 \geq j_1$. In (1), (2), and (4), we can take $m_{j'-1}=k-1, m_{j'}=k$ for the terms of $E^\bullet$. Then, clearly relevant Ext-groups does not vanish. To show (3), we set $E^\bullet$ to be a complex  
        $$
            \cdots \to 0 \to \cO(k'-1)^{\oplus n'_{j'-1}} \to \cO(k')^{\oplus n'_{j'}} \to 0 \to \cdots
        $$
        whose terms on degrees $j_2-1$ and $j_2$. From the above, we have only non-trivial $\rExt$-groups 
        \begin{equation}
            \rExt^{i+j_2-(j_1+1)}(\cO(k-e-1),\cF) \quad and  \quad \rExt^{i+j_2-j_1}(\cO(k-e),\cF) \label{eq.prop3.1-1}
        \end{equation}
        when $i+j_2-j_1=0,1$ where $\cF \in \{ \cO(k'-1)^{\oplus n'_{j_2-1}}, \cO(k')^{\oplus n'_{j_2}}\}$. Thus, it suffices to show that these do not vanish simultaneously when $j_2 \geq j_1$. Set $E = \cO(k')$, and assume that $\rExt^0(\cO(k-e-1),\cO(k'))$ and $\rExt^1(\cO(k-e),\cO(k'))$ are vanish simultaneously. Then it must be $k-e-1 > k'$ from the left term. On the otherhands, by applying the Serre duality to the right term, we have $\rExt^1(\cO(k-e),\cO(k')) = \rExt^0(\cO(k'),\cO(k-e-2))^\vee$. From this, it must hold $k' > k-e-2$. Combining these, we obtain an inequality $k-e-1 > k' > k-e-2$, but no such an integer $k'$ exist. This is contradiction. The same argument leads to a contradiction in the case of $E=\cO(k'-1)$. Thus (3) holds true.
\end{proof}
\end{proposition}

\begin{proposition}
    Let $\cA(k)[j] = \langle \cO(k-1)[j+1], \cO(k)[j] \rangle$ be a heart of bounded t-structure of Ext-exceptional collection $\{ \cO(k-1)[j+1], \cO(k)[j] \}$ in $D^b(\bP^1)$. Then $\cA(k)[j]$ includes $\bigoplus_j  \cO(n_j)[j+1]^{\oplus m_j}$ for all $n_j < k-1$ and $m_j \geq 0$. Similaly, $\cA(k)[j]$ includes $\bigoplus_j  \cO(n_j)[j]^{\oplus m_j}$ for all $n_j > k$ and $m_j \geq 0$. Moreover the skyscraper sheaf $\cO_x$ is included in $\cA(k)[j]$.
    \begin{proof}
        First, assume that $\cA(k)[1] = \langle \cO(k-1)[1], \cO(k) \rangle$. From \cite{Oka05} Lemma 3.1(2), 
        for all $n < k-1$ there are triangle in $D^b(\bP^1)$
        $$
            \cO(k)^{\oplus k-1-n}[-1] \to \cO(n) \to \cO(k-1)^{\oplus k-n}.
        $$
        After the shifts, we obtain
        \begin{equation}         
            \cO(k)^{\oplus i-1-n} \to \cO(n)[1] \to \cO(k-1)[1]^{\oplus i-n}. \label{trignale1}
        \end{equation}
        Obviously, the left and right terms of this triangle are the object of $\cA(k)[1]$. Hence by the extension clossedness of $\cA(k)[1]$, the middle term also included in $\cA(k)[1]$.

        Assume that $N = \{ n \mid n < k-1 \}$ and take $n_1,n_2 \in N$. Clearly, there exists a triangle
        $$
            \cO(n_1)[1] \to \cO(n_1)[1] \oplus \cO(n_2)[1] \to \cO(n_2)[1].
        $$
        By the above arguments, the left and right terms of above triangle are the object of $\cA(k)[1]$. Hence $(\cO(n_1) \oplus \cO(n_2))[1] \in \cA(k)[1]$. Inductively, $\bigoplus_j  \cO(n_j)[1]^{\oplus m_j}$, $n_j \in N$ is an object of $\cA(k)[1]$ for $m_j \geq 0$. Obviously, the similar argument holds for $\cA(k)[j]$ by applied the $j$-times shift to a triangle \eqref{trignale1}.

        The second statement is holds by applying the same argument as above to the following triangle obtained by \cite{Oka05} Lemma 3.1 (1) : for all $n > k$, there are triangle
        $$
            \cO(k)^{\oplus n-k+1} \to \cO(n) \to \cO(k-1)^{\oplus n-k}[1].
        $$

        The last statement is established by triangle 
        \begin{equation} \label{eq.Oka05Lemma3.1(3)}            
            \cO(k) \to \cO_x \to \cO(k-1)[1]
        \end{equation}
        in \cite{Oka05} Lemma 3.1 (3).
    \end{proof}
\end{proposition}

 We use the following four type notations.
\begin{definition} \label{typeofgluingstab}
    Let $\sigma_{gl} = (Z_{gl}, \cA_{gl})$ be a gluing stability condition on $D^b(\hir)$ such that $\cA_{gl}$ is a glued heart constructed in Definition \eqref{gluableheart} and $Z_{gl}$ is a stability function on $\cA_{gl}$ defined in Definition \ref{gluingstabilitycondition} \eqref{glfunc}. Then we use the following notation:
    \begin{enumerate} 
        \item $\sigma_{gl,1}$ if $\sigma_{gl}$ glued from $p^*\rCoh(\bP^1)[j_1] \otimes \cO_{\hir}(-C_0)$ and $p^*\cA(k)[j_2]$ \label{gl1}
        \item $\sigma_{gl,2}$ if $\sigma_{gl}$ glued from $p^*\cA(k)[j_1] \otimes \cO_{\hir}(-C_0)$ and $p^*\rCoh(\bP^1)[j_2]$ \label{gl2}
        \item $\sigma_{gl,3}$ if $\sigma_{gl}$ glued from $p^*\cA(k)[j_1] \otimes \cO_{\hir}(-C_0)$ and $p^*\cA(k')[j_2]$ \label{gl3}
        \item $\sigma_{gl,4}$ if $\sigma_{gl}$ glued from $p^*\rCoh(\bP^1)[j_1] \otimes \cO_{\hir}(-C_0)$ and $p^*\rCoh(\bP^1)[j_2]$ \label{gl4}
    \end{enumerate}
    We call $\sigma_{gl,m}$ the {\it glued type $m$}.
\end{definition}

Next, we observe the stability functions on the heart of a bounded t-structure which constructed in above. The derived category $D^b(\bP^1)$ has the following two types of stability conditions:
\begin{itemize}
    \item ({\bf Standard stability condition}) 
    \begin{equation}
        \sigma_{st} := (Z_{st}, \rCoh(\bP^1)), \label{standardstabilitycondition}
    \end{equation} where the stability function $Z_{st}$ is given by $-\rdeg + i \rrank$. \vspace{1mm}
    \item  ({\bf Quiver stability condition}) 
    \begin{equation}
        \sigma_{k,\zeta_0,\zeta_1} := (Z_{\zeta_0,\zeta_1},\cA(k)[j]), \label{algebraicstabilitycondition}
    \end{equation} where the stability function $Z_{\zeta_0,\zeta_1} : K(\bP^1) \to \bC$ is defined by fixing two complex numbers $\zeta_0,\zeta_1 \in \bH$ and setting: 
    \begin{equation}    
          Z_{\zeta_0,\zeta_1}(\cO(k-1)[j+1]) = \zeta_0, \quad Z_{\zeta_0,\zeta_1}(\cO(k)[j]) = \zeta_1.  \label{eq.alg.stab.func.}
    \end{equation}
\end{itemize}
\hspace{1mm}

Since $\cA(k)[j]$ generated by the Ext-exceptional collection, $\sigma_{k,\zeta_0,\zeta_1}$ is a stability condition on $D^b(\bP^1)$. Moreover, the generators $\cO(k-1)[j+1]$ and $\cO(k)[j]$ are $\sigma_{k,\zeta_0,\zeta_1}$-stable (\cite{Mac07}). From the construction \eqref{IndStabD1} and \eqref{IndStabD2}, these stability conditions can be pulled back to induce stability conditions on the components of the semiorthogonal decomposition of $D^b(\Sigma_e)$.

\begin{lemma} \label{lambda1rho2}
    For any $E \in D^b(\hir)$, the adjoint functors $\lambda_1 : D^b(\hir) \to p^*D^b(\bP^1) \otimes \cO_{\hir}(-C_0)$ and $\rho_2 : D^b(\hir) \to p^*D^b(\bP^1)$ can be written as follows:
    \begin{enumerate}
        \item $\lambda_1(E) = p^*(\bfR p_* E(-C_0 +(-e-2)f) \otimes \cO(2)) \otimes \cO_{\hir}(-C_0) [1]$,
        \item $\rho_2(E) = p^*\bfR p_*E$
    \end{enumerate}
    where $e$ is the degree of $\Sigma_e$.
    \begin{proof}
        Applying \cite{Uch} Lemma 3.3 to the case where $g(C)=0$.    \end{proof}
\end{lemma}

\begin{remark}\label{rmk.degreeofshiftofgluedheart}
    Notice that, $\lambda_1$ and $\rho_2$ determines the degree of shift of the glued heart. Namely, if $\cA_{gl}$ glued from $\cA_1[j_1]$ and $\cA_2[j_2]$ in \eqref{typeofgluingstab}, then $j_1 = j_2+1$ holds. Indeed, the difference of the shift degrees of $\lambda_1$ and $\rho_2$ are $1$ by the Grothendieck-Verdier duality. 
\end{remark}

To compute the glued stability function $Z_{gl} = Z_1 \circ \lambda_1 + Z_2 \circ \rho_2$  \eqref{glfunc}, we need to determine the rank and degree of the images of $\lambda_1$ and $\rho_2$.
\begin{proposition} \label{Prop_rankdegofadjointfunctor}
    Let $\lambda_1$ and $\rho_2$ are the functors which constructed in the lemma \ref{lambda1rho2}. We write $\rch(E) = (r,c_1,\rch_2)$ for the Chern class of $E \in D^b(\hir)$. Then
    $$
        \rch_0(\lambda_1(E)) = -c_1.f, \quad \rch_1(\lambda_1(E)) = - \rch_2 +\frac{1}{2}e c_1.f.
    $$
    Simirarly, 
    $$
        \rch_0(\rho_2(E)) = c_1.f+r, \quad \rch_1(\rho_2(E)) = \rch_2+c_1.C_0+\frac{1}{2}ec_1.f.
    $$
    \begin{proof}
        All of these are obtained by simple calculation of Chern characters. Notice that, $C_0^2=-e, C_0.f = 1$, and $f^2=0$ by the intersection theory. Since a functor $p^*$ and $-\otimes \cO_{\hir}(-C_0)$ induces an isomorphism of Grothendieck groups $K(D^b(\bP^1)) \cong K(p^*D^b(\bP^1) \otimes \cO_{\hir}(-C_0))$ (\cite{Uch} \S 3), we can reduce to caluculation of $\rch(Rp_*E(-C_0+(-e-2)f) \otimes \cO_{\bP^1}(2))[1]$. By the Grothendieck-Riemann-Roch formula, we have
        \begin{eqnarray*}
          \rch(\lambda_1(E)) &=& - \left( \frac{p_*(\rch(E)\rch(D)\rtd(\hir))}{\rtd(\bP^1)} \right) \rch(2H) \\ 
          &=& (-c_1.f,  - \rch2 + \frac{1}{2}ec1.f).
        \end{eqnarray*}
        where $H$ is an ample divisor on $\bP^1$ and $D = -C_0+(-e-2)f$. Similarly, we can calculate $\rch(\rho_2(E))$ as 
        \begin{eqnarray*}
          \rch(\rho_2(E)) &=& \frac{p_*(\rch(E)\rtd(\hir))}{\rtd(\bP^1)} \\ 
          &=& (c_1.f+r, \rch_2 + c_1.C_0+\frac{1}{2}ec_1.f).
        \end{eqnarray*}
    \end{proof}
\end{proposition}

From now on, we use the notation that, 
\begin{equation} \label{rankanddegreeoflambda1andrho2}
    r_{\lambda_1} = \rch_0(\lambda_1(F)), \quad d_{\lambda_1} = \rch_1(\lambda_1(F)), \quad r_{\rho_2} = \rch_0(\rho_2(F)), \quad d_{\rho_2} = \rch_1(\rho_2(F)).
\end{equation}
As a consequence, for the standard stability condition, we have 
$$
    Z_{st} \circ \lambda_1 = -d_{\lambda_1} + i r_{\lambda_1} \quad \textrm{and} \quad Z_{st} \circ \rho_2 = -d_{\rho_2} + i r_{\rho_2}.
$$ 

The stability functions for quiver stability conditions are different from those of standard stability conditions and must be handled with care. Let $\sigma_{k,\zeta_0,\zeta_1} = (Z_{\zeta_0,\zeta_1},\cA(k)[j])$ be a quiver stability condition as constructed in \eqref{algebraicstabilitycondition}. A general object of $\cA(k)[j]$ can be represented by a two-term complex of the form:
$$
    E = \bC^{n_0} \otimes_\bC \cO(k-1) \to \bC^{n_1} \otimes_\bC \cO(k),
$$
concentrated in degrees $j-1$ and $j$. The pair $[n_0,n_1]_E$ is called the {\it dimension vector} of $E$. Notice that, when the image of an object under a functor $\lambda_1$ or $\rho_2$ lies within such a heart, it can also be described by such complexes and their corresponding dimension vectors. By construction, we have
\begin{eqnarray*} \label{definitionofZijp}
    Z_{\zeta_0,\zeta_1}(E) & = & n_0(E) \zeta_0 + n_1(E) \zeta_1  
\end{eqnarray*}
for some $\zeta_0,\zeta_1 \in \bH$.

The above can be summarized as follows.
\begin{proposition} \label{prop.gluingstabilityfunction}     
    Suppose that $\cA_{gl,m}$ be a glued heart of bounded t-structure which defined in the \eqref{Gluingheart}. Then the following $Z_{gl,m}$ admits the stability function on $\cA_{gl,m}$.
    \begin{enumerate}
        \item When $m=1$, \begin{equation*}
                  Z_{gl,m}(E)  = (-1)^{j_1}(- d_{\lambda_1}(E)  + ir_{\lambda_1}(E)))  + (-1)^{j_2}(n_0(E)\zeta_0 + n_1(E) \zeta_1)
              \end{equation*}
        \item When $m=2$, \begin{equation*}
                   Z_{gl,m}(E)  = (-1)^{j_2}(-d_{\rho_2}(E) + ir_{\rho_2}(E)) + (-1)^{j_1}(n_0(E) \zeta_0 + n_1(E) \zeta_1), \\
              \end{equation*}
        \item When $m=3$, \begin{equation*}
                   Z_{gl,m}(E)  = (-1)^{j_1}(n_0(E) \zeta_0 + n_1(E) \zeta_1) + (-1)^{j_2}(n_0'(E) \zeta_0' +  n_1'(E) \zeta_1'), \\
              \end{equation*}
        \item When $m=4$, \begin{equation*}
                    Z_{gl,m}(E)  = (-1)^{j_1}(-d_{\lambda_1}(E) + i r_{\lambda_1}(E)) + (-1)^{j_2}(- d_{\rho_2}(E)+i r_{\rho_2}(E)).
              \end{equation*}
    \end{enumerate}
    Here, $n_0(E)$ and $n_1(E)$ are the dimension vectors of the image of $E$ under the appropriate functor ($\lambda_1$ or $\rho_2$) into the heart of the quiver stability condition. In particular, $\sigma_{gl,m} = (Z_{gl,m},\cA_{gl,m})$ is a pre-stability condition on $D^b(S)$ for $m =1,2,3,4$.
    \begin{proof}
        We need to show that each $Z_{gl,m}$ satisfies the Harder-Narasimhan property (\cite{CP} \S 2). By construction, for each $m$, the heart $\cA_{gl,m}$ is a gluing of two hearts $\cA_1$ and $\cA_2$ that satisfy the gluing condition $\rhom^{\leq 0}(\cA_1,\cA_2) = 0$. Furthermore, for each stability functions $Z_1$ and $Z_2$ used to construct $Z_{gl,m}$ already have the Harder-Narasimhan property. The result then follows directly from [\cite{CP} Proposition 3.3].
    \end{proof}
\end{proposition}

Notice that, these $\sigma_{gl,m} = (Z_{gl,m},\cA_{gl,m})$ are numerical stability condition. 
\begin{lemma}
    A stability condition $\sigma_{gl,m} \in \cG l(\sigma_1,\sigma_2)$ is a numerical stability condition for $m=1,2,3,4$. 
    \begin{proof}
        We show that a stability function $Z_{gl,m}$ of $\sigma_{gl,m}$ factors through $K_{num}(\Sigma_e)$. By definition, we have $Z_{gl,m} = Z_1 \circ \lambda_1 + Z_2 \circ \rho_2$. By the additivity of $K_{num}(\Sigma_e)$, it is sufficient to show that both $Z_{gl,m}|_{\cA_i} = Z_i \circ \Phi$ factors through $K_{num}(\Sigma_e)$ for $i=1,2$, where $\Phi \in \{ \lambda_1,\rho_2\}$. Note that $Z_{gl,m}|_{\cA_i}$ is a group homomorphism $Z_{gl,m}|_{\cA_i} : K(F(D^b(\bP^1)) \cong K(D^b(\bP^1)) \to \bC$ where $F \in \{ p^*, (-\otimes\cO_S(-C_0)) \circ p^* \}$ from \eqref{isomofGrothendieckgroups}. But, then we have an isomorphism $K(F(D^b(\bP^1)) \cong K_{num}(F(D^b(\bP^1)))$ since there exists a canonical isomorphism $K(\bP^1) \cong K_{num}(\bP^1) \cong \bZ^{\oplus 2}$. Therefore $Z_i \circ \Phi$ factors through $K_{num}(\Sigma_e)$ via the compotision
        $$
            K(F(D^b(\bP^1))) \cong K_{num}(F(D^b(\bP^1)) \hookrightarrow K_{num}(\Sigma_e).
        $$
        
    \end{proof}
\end{lemma}

\begin{remark}
    For a derived category $\cD$ of an abelian category $\cA$ which generated by an exceptional collection $E_1,...,E_n$, it is well-known fact that $(K(\cA) \cong) \ K(\cD) \cong \bZ^{\oplus n}$, and hence in the our case $K(\cA) \cong \bZ^{\oplus 2}$. Since the Euler form of $K(\bP^1)$ is non-degenarate, this isomorphism lift up to $K_{num}(\bP^1) \cong \bZ^{\oplus 2}$. 
\end{remark}

Next, we show that these four types of pre-stability conditions satisfy the support property. To this end, we first require some preliminary results for the cases $m=1, 2, 3, 4$. In the cases $m=1$ and $3$ (where $\sigma_2$ is a quiver stability condition), the following criterion applies. 
\begin{lemma}[\cite{karube2024noncommutativemmpblowupsurfaces} Proposition 3.11] \label{karube2024noncommutativemmpblowupsurfaces-prop3.11}
    Let $\sigma_1 = (Z_1,\cA_1)$ and $\sigma_2 = (Z_2,\cA_2)$ be stability conditions on $\cD_1$ and $\cD_2$ respectively. We assume that $\rhom^{\leq 0}(\cA_1,\cA_2) = 0$. Assume that there exists the pre-stability condition $\sigma=(Z,\cA)$ glued from $\sigma_1$ and $\sigma_2$. If there exists $0 < \theta \leq 1$ such that $Z(\cA_2) \subset \bH_{\theta}$, where $\bH_\theta$ is the set defined by
    $$
        \{ r e^{i\pi\phi} \mid r \in \bR_{>0}, \phi \in [\theta,1]\}, 
    $$
    then the glued pre-stability condition $\sigma=(Z,\cA)$ satisfies the support property.
\end{lemma}

The case $m=2$ is slightly more involved; to address this, we establish the following technical lemma. We define 
$$
    \theta_q := {\rm min} \{ \phi(\zeta_0),\phi(\zeta_1) \}.
$$
\begin{lemma} \label{inequalityofstabilityfunctions}
    Let $\sigma = (Z_{gl,2},\cA_{gl,2})$ be a gluing pre-stability condition glued from a quiver stability condition $\sigma_1 = \sigma_{k,\zeta_0,\zeta_1}$ and a standard stability condition $\sigma_2=\sigma_{st}$ for some $k \in \bZ$ and $\zeta_0,\zeta_1 \in \bH$. Then there exists a constant $C > 0$ such that for any $\sigma$-semistable object $E$, 
    \begin{equation}
        |Z_1(E)| \leq C |Z_{gl,2}(E)| \quad \textrm{and} \quad |Z_2(E)| \leq (1+C)|Z_{gl,2}(E)|
    \end{equation}

    \begin{proof}
    Let us denote $z_1(E) := Z_1(\lambda_1(E))$ and $z_2(E) := Z_2(\rho_2(E))$. Note that $Z_{gl,2}(E) = z_1(E) + z_2(E)$.
    From the definition of the quiver stability condition, the argument of $z_1(E)$ satisfies
    $$
        \arg(z_1(E)) \in [\arg(\zeta_0), \arg(\zeta_1)] \subset [\pi \theta_q, \pi(1-\theta_q)].
    $$
    Consequently, we have
    $$
        \Im z_1(E) = |z_1(E)| \sin (\arg(z_1(E))) \geq |z_1(E)| \sin(\pi \theta_q).
    $$
    
    Since $\sigma_2$ is the standard stability condition, we have $\Im z_2(E) \ge 0$. By the additivity of the central charge, it follows that
    $$
        \Im Z_{gl,2}(E) = \Im z_1(E) + \Im z_2(E) \geq \Im z_1(E).
    $$
    Using the basic inequality $|w| \ge \Im w$, we obtain
    $$
        |Z_{gl,2}(E)| \geq \Im Z_{gl,2}(E) \geq \Im z_1(E) \geq |z_1(E)| \sin(\pi \theta_q).
    $$
    Setting $C := \frac{1}{\sin(\pi \theta_q)}$, we have
    \begin{equation}
        |z_1(E)| \leq C |Z_{gl,2}(E)|.
    \end{equation}
    Finally, combining this with the triangle inequality, we deduce the desired bound:
    $$
        |z_2(E)| = |Z_{gl,2}(E) - z_1(E)| \leq |Z_{gl,2}(E)| + |z_1(E)| \leq (1+C)|Z_{gl,2}(E)|.
    $$
    \end{proof}
\end{lemma}

The case $m=4$ is technically more difficult. Unlike the quiver stability case, 
a lower bound on the phase is not automatically ensured. We postpone the detailed proof to Section \ref{supportpropertyofm=4}.
\begin{proposition} \label{supportpropertyofgluingstabilitycondition}
    For each $m=1,2,3,4$, any glued pre-stability condition $\sigma_{gl,m} = (Z_{gl,m},\cA_{gl,m}) \in \cG l(\sigma_1,\sigma_2)$ satisfies the support property. In particular, each $\sigma_{gl,m}$ is a locally finite stability condition.
   \begin{proof}
        For the case $m=4$ follows from the Proposition \ref{prop:m4-support-property}. We consider the cases $m=1$ and $3$ (where $\cA_2$ is the heart of a quiver stability condition). We apply the criterion in Lemma \ref{karube2024noncommutativemmpblowupsurfaces-prop3.11}. Let $E \in \cA_2$ be a non-zero object. Its central charge is given by $Z_2(E) = n_0 \zeta_0 + n_1 \zeta_1$, where $\zeta_0,\zeta_1 \in \bH$, and the dimension vector $(n_0,n_1)$ consists of non-negative integers. Since $E \neq 0$, the dimension vector is non-zero, meaning at least one of $n_0$ or $n_1$ is a positive integer. Observe that the phase of the sum $n_0 \zeta_0 + n_1 \zeta_1$ is bounded by the phases of $\zeta_0$ and $\zeta_1$. We define $\theta := \min \{ \phi(\zeta_0), \phi(\zeta_1) \}$. Since $\zeta_0,\zeta_1 \in \bH$, the image $Z_2(\cA_2 \setminus \{0\})$ is contained in the sector $\bH_\theta = \{ r e^{i \pi \phi} \mid \phi \in [\theta,1] \}$. Thus, the condition of Lemma \ref{karube2024noncommutativemmpblowupsurfaces-prop3.11} is satisfied, and the support property holds.

        Next, we consider the case $m=2$. Recall that $\sigma_1 = \sigma_{k,\zeta_0,\zeta_1} = (Z_1,\cA_1)$ and $\sigma_2 = \sigma_{st} = (Z_2,\cA_2)$ satisfy the support property. That is, for $i=1,2$, there exists a constant $C_i > 0$ such that
        $$
            \|v_i(E)\|_i \leq C_i |Z_i(E)|
        $$
        for all $\sigma_i$-semistable objects $E$, where $\|\cdot\|_i$ is a norm on $\Lambda_i \otimes \bR$. Using the isomorphism $\Lambda \cong \Lambda_1 \oplus \Lambda_2$, we equip $\Lambda_\bR$ with the norm $\|v\| := \|v_1\|_1 + \|v_2\|_2$.
        Thus, for any object $E$, we have
        $$
            \| v(E) \| \leq \| v_1(\lambda_1(E)) \|_1 + \| v_2(\rho_2(E)) \|_2.
        $$
        Combining these with the support property of each component, we obtain
        $$
            \| v(E) \| \leq C_1 |Z_1(\lambda_1(E))| + C_2 |Z_2(\rho_2(E))|.
        $$
        By Lemma \ref{inequalityofstabilityfunctions}, there exists a constant $C' > 0$ such that the above inequality becomes
        $$
            \| v(E) \| \leq C_1 C' |Z_{gl,2}(E)| + C_2(1+C') |Z_{gl,2}(E)|.
        $$
        Setting $C = C_1 C' + C_2(1+C')$, we obtain the desired inequality
        $$
            \| v(E) \|\leq C |Z_{gl,2}(E)|.
        $$
\end{proof}
\end{proposition}
        
For later use, we discuss alternative representations for each $Z_{gl,m}$. According to \cite{arcara2013minimal} \S 7, There are transformation matrices between Chern classes and dimension vectors. The conversion from a dimension vector to a vector of Chern class is given by 
  \begin{equation}
      C = \begin{pmatrix}
          -1 & 1 \\
          1-k & k
      \end{pmatrix}.
  \end{equation}
Conversely, the conversion from a vector of Chern class to a dimension vector is given by
  \begin{equation}
      C^{-1} = \begin{pmatrix}
          -k & 1 \\
          1-k & 1
      \end{pmatrix}.
  \end{equation}
Therefore, for example, in the case where $j=0$ in $\cA(k)[j]$, for an object $E \in \cA(k)$ with dimension vector $[n_0,n_1]_E$, its rank and degree are computed as 
 \begin{equation} \label{convDimvectToChern}     
    \begin{pmatrix}
        \rrank \ E \\
        \rdeg \ E
    \end{pmatrix} =
      C \begin{pmatrix}
          n_0 \\
          n_1
      \end{pmatrix} = \begin{pmatrix}
          n_1-n_0 \\
          (1-k)n_0+kn_1
      \end{pmatrix}.
 \end{equation}
Conversely, since the degree and rank of $\lambda_1$ and $\rho_2$ are given by proposition \ref{Prop_rankdegofadjointfunctor}, the dimension vector for the image of an object can be calculated as
\begin{equation} \label{dimensionvectoroflambda1andrho2}
    \begin{pmatrix}
        n_0 \\
        n_1
    \end{pmatrix} =
     \begin{pmatrix}
         -k & 1 \\ 1-k & 1
     \end{pmatrix} 
       \begin{pmatrix}
          r_{*} \\
          d_{*}
      \end{pmatrix} = \begin{pmatrix}
           d_{*} - k r_{*}\\
          d_{*} + (1-k)r_{*}
      \end{pmatrix}.
 \end{equation}
where $* \in \{ \lambda_1, \rho_2 \}$. This gives a conversion of the coordinates on the Grothendieck groups 
$$
    K(\bP^1) \xrightarrow{(\rrank,\rdeg )} \bZ^{\oplus 2} \ \ \underset{C}{\overset{C^{-1}}{\rightleftarrows}} \ \ \bZ^{\oplus 2} \xleftarrow{(n_0,n_1)} K(\cA(k)).
$$
Considering that $K_{num}(\bP^1) \cong \mathbb{Z}^{\oplus 2}$, this means that the stability function for a quiver stability condition $\sigma$ can be expressed as a group homomorphism of the numerical Chern characters via this transformation. Hence, $\sigma$ is a numerical stability condition. As a result, the stability function $Z_{gl,m}$ defined in Proposition \ref{prop.gluingstabilityfunction} can be rewritten in terms of Chern characters as follows:
\begin{align} \label{eq.stabilityfunctionrepresentedbyCherncharacter}
    Z_{gl,1} = (-1)^{j_1}(&- d_{\lambda_1}(E)  + ir_{\lambda_1}(E))  + (-1)^{j_2}((d_{\rho_2}-k r_{\rho_2} )\zeta_0 + (d_{\rho_2} + (1-k) r_{\rho_2})\zeta_1) \notag \\ 
    Z_{gl,2} = (-1)^{j_2}(&-d_{\rho_2}(E) + ir_{\rho_2}(E)) + (-1)^{j_1}((d_{\lambda_1}-k r_{\lambda_1} )\zeta_0 + (d_{\lambda_1} + (1-k) r_{\lambda_1})\zeta_1) \notag \\ 
    Z_{gl,3} = (-1)^{j_1}(&(d_{\lambda_1}-k r_{\lambda_1} )\zeta_0 + (d_{\lambda_1} + (1-k) r_{\lambda_1})\zeta_1) \notag \\ 
    &+(-1)^{j_2}((d_{\rho_2}-k r_{\rho_2} )\zeta'_0 + (d_{\rho_2} + (1-k) r_{\rho_2})\zeta'_1).
\end{align}
\hspace{1mm}

\subsection{The support property of $m=4$} \label{supportpropertyofm=4}
In this section, we establish the support property for the glued type $m=4$. Recall that for any object $E \in \cA_{gl,4}$, 
there exist $E_1, E_2 \in \rCoh(\bP^1)$ and a canonical short exact sequence
\begin{equation} \label{eq.exactseqofAgl4}
    0 \to \rho_2(E)=p^*E_2 \to E \to p^*E_1 \otimes \cO_{\Sigma_e}(-C_0)[1] =\lambda_1(E) \to 0
\end{equation}
(cf. \eqref{gluingexactseq}). Combining the additivity of Chern characters with this sequence yields
\begin{align}
    \rch(E) &=  \rch(p^*E_2) + \rch(p^*E_1 \otimes \cO_{\Sigma_e}(-C_0)[1]) \notag \\
            &= (r_2-r_1, r_1C_0+(d_2-d_1)f, d_1+\frac{1}{2}e r_1) \notag \\
            &=: (r,c_1,\rch_2) \label{m4chern}
\end{align}
where $r_i:=\rrank(E_i)$ and $d_i:=\rdeg(E_i)$ for $i=1,2$. On the other hand, from Proposition \ref{prop.gluingstabilityfunction} (4), Proposition \ref{rankanddegreeoflambda1andrho2}, and Remark \ref{rmk.degreeofshiftofgluedheart}, the stability function $Z_{gl,4}$ can be written as
\begin{align}
    Z_{gl,4} &= (-1)^{j_1}(-d_{\lambda_1} +ir_{\lambda_1}) + (-1)^{j_2}(-d_{\rho_2} + ir_{\rho_2} ) \notag \\
    &= d_{\lambda_1} - d_{\rho_2} + i(r_{\rho_2} - r_{\lambda_1}) \notag \\
             &= \left(  - \rch_2 +\frac{1}{2}e c_1.f\right) - \left( \rch_2+c_1.C_0+\frac{1}{2}ec_1.f \right) + i \left( c_1.f+r -(-c_1.f) \right) \notag \\
             &= -(2\rch_2+c_1C_0) + i(r + 2c_1.f). \label{m4stabilityfunc}
\end{align}
Here, we used the relation $j_1 = j_2+1$ (Remark \ref{rmk.degreeofshiftofgluedheart}) for the second equality. Substituting the Chern characters from \eqref{m4chern} into \eqref{m4stabilityfunc}, we obtain the simplified equation:
\begin{equation} \label{m4stabilityfuncasstandard}
    Z_{gl,4}(E) = -(d_1+d_2) + i(r_1+r_2) = Z_{st}(E_1) + Z_{st}(E_2).
\end{equation}
This is expressed as the sum of two standard stability functions $Z_{st} = -\rdeg +i \rrank$ on the heart $\rCoh(\bP^1)$. Hence we define 
\begin{equation} \label{slopeofZgl4}
    \mu(E) :=  \frac{d_1+d_2}{r_1+r_2}
\end{equation}
and let us define the values
$$
    D := d_1+d_2 = \rRe Z_{gl,4}(E), \qquad R:= r_1+r_2 = \rIm Z_{gl,4}(E).
$$
From the above, we see that $\mu(\lambda_1(E)) = \frac{d_1}{r_1}$ and $\mu(\rho_2(E)) = \frac{d_2}{r_2}$ coincide with the slopes $\mu(E_1)$ and $\mu(E_2)$ on the base curve. Accordingly, we denote the maximal and minimal slopes by $\mu_{\max}$ and $\mu_{\min}$, respectively.

From equation \eqref{m4stabilityfuncasstandard}, the absolute value of $Z_{gl,4}$ satisfies
\begin{equation}\label{eq:m4-absZ}
    |Z_{gl,4}(E)|=\sqrt{(d_1+d_2)^2+(r_1+r_2)^2}\ \ge\ |d_1+d_2|,\ r_1+r_2.
\end{equation}

In order to prove the support property for the case $m=4$, we first require some preliminary results.
\begin{lemma} \label{lem:m4-torsion-exclusion}
    Let $0\ne E\in\cA_{gl,4}$ be a $\sigma_{gl,4}$-semistable object and write it as in \eqref{eq.exactseqofAgl4}. Then:
    \begin{enumerate}
        \item If $R>0$, then both $E_1$ and $E_2$ are torsion-free, hence vector bundles on $\bP^1$.
        \item If $R=0$, then $E_1$ and $E_2$ are torsion sheaves and $Z_{gl,4}(E)=-(d_1+d_2)\in\bR_{<0}$.
    \end{enumerate}

    \begin{proof}
        (1) Assume $R>0$. If $E_2$ has a nonzero torsion subsheaf $T\subset E_2$, then
        $p^*T\subset p^*E_2 =\rho_2(E)\subset E$ is a subobject in $\cA_{gl,4}$. Since $Z_{st}(T)\in\bR_{<0}$, 
        we have $p^*T\in\cP(1)$, hence $\phi(p^*T)=1$. On the other hand, $\rIm Z_{gl,4}(E)=R>0$ implies $\phi^+(E)<1$. Thus $\phi(p^*T)>\phi^+(E)$, contradicting the $\sigma_{gl,4}$-semistability of $E$.
        
  Next, suppose $E_1$ has a nonzero torsion subsheaf $T \subset E_1$. Then $p^*T \otimes \cO_{\Sigma_e}(-C_0)[1]$ is a subobject of $\lambda_1(E)$ in $\cA_{gl,4}$. Since $\rho_2(p^*T \otimes \cO_{\Sigma_e}(-C_0)[1]) = 0$, this object lifts to a subobject $S \subset E$ in $\cA_{gl,4}$.
The central charge of $S$ is $Z_{gl,4}(S) = Z_{st}(T) = -\mathrm{deg}(T) \in \mathbb{R}_{<0}$, which implies $\phi(S)=1$.
On the other hand, since $\Im Z_{gl,4}(E) = R > 0$, we have $\phi(E) < 1$. The existence of a subobject $S$ with $\phi(S) > \phi(E)$ contradicts the semistability of $E$. Thus, $E_1$ and $E_2$ must be torsion-free. On $\bP^1$, they are vector bundles.
        
        (2) If $R=0$, then $r_1=r_2=0$, so $E_1$ and $E_2$ are torsion, and $Z_{gl,4}(E)=-(d_1+d_2)\in\bR_{<0}$ follows from the definition.
    \end{proof}
\end{lemma}

\begin{lemma} \label{lem:m4-degree-control}
    Let $0\ne E\in\cA_{gl,4}$ be $\sigma_{gl,4}$-semistable and write it as in \eqref{eq.exactseqofAgl4}. Assume $R=r_1+r_2>0$, so $E_1$ and $E_2$ are vector bundles by Lemma~\ref{lem:m4-torsion-exclusion}. Let $\mu(E)=\frac{d_1+d_2}{r_1+r_2}$ defined in \eqref{slopeofZgl4}. Then the following hold.
    \begin{enumerate}
        \item One has the inequality
            \begin{equation}
                \mu_{\max}(E_2)\ \le\ \mu(E)\ \le\ \mu_{\min}(E_1).
            \end{equation}
    
        \item If $E$ is non-split and $e>0$, then $\rhom_{D^b(\bP^1)}(E_1,E_2)\ne 0$.
    
        \item In either case (split or non-split), $E_1$ is slope-semistable and
                \begin{equation} \label{eq:m4-d1-linear}
                    d_1=\mu(E)\,r_1=\frac{r_1}{r_1+r_2}\,(d_1+d_2).
                \end{equation}
              In particular,
                \begin{equation} \label{eq:m4-d-bounds}
                    |d_1|\le |d_1+d_2|,\qquad |d_2|\le 2|d_1+d_2|.
                \end{equation}
    \end{enumerate}
    \begin{proof}
        (1) The existence of the subobject $\rho_2(E)\subset E$ and the quotient $E\twoheadrightarrow \lambda_1(E)$ forces the phase inequalities. Since $Z_{st}=-\rdeg+i\rrank$ on $\bP^1$, these inequalities are equivalent to the HN-inequalities for slope. 

        (2) By the projection formula and canonical isomorphism $ \bfR p_*\cO_{\Sigma_e}(C_0) \cong \cO \oplus \cO(-e)$, 
        we have an isomorphism
         \begin{align*}
                     \rExt^1_{\cA_{gl,4}}(\lambda_1(E),\rho_2(E)) &\cong 
                     \rhom_{D^b(\bP^1)}(E_1,E_2 \otimes (\cO \oplus \cO(-e))\\
                     &\cong \rhom_{D^b(\bP^1)}(E_1,E_2) \oplus \rhom_{D^b(\bP^1)}(E_1,E_2(-e)),
         \end{align*} 
         where the first isomorphism is given by \eqref{projectionformula1}. 
        Suppose $\rhom(E_1, E_2(-e)) \neq 0$. Then by the stability of vector bundles on curves, we have $\mu_{\min}(E_1) \le \mu_{\max}(E_2(-e)) = \mu_{\max}(E_2) - e$. Combining this with (1), we obtain
        \begin{equation*}
            \mu_{\max}(E_2) \le \mu(E) \le \mu_{\min}(E_1) \le \mu_{\max}(E_2) - e.
        \end{equation*}
        This implies $e \le 0$, which contradicts the assumption $e > 0$. Consequently, $\rhom(E_1, E_2(-e)) = 0$. Since $E$ is non-split, the extension class implies a non-zero element in the remaining summand, i.e., $\rhom(E_1, E_2) \neq 0$.
        
        (3) First assume $E$ is non-split and $e>0$. From (2) we have $\rhom(E_1,E_2)\ne 0$, which implies $\mu_{\min}(E_1) \le \mu_{\max}(E_2)$. Combining this with the inequality from (1), we obtain a chain of inequalities:
\[
    \mu_{\max}(E_2) \le \mu(E) \le \mu_{\min}(E_1) \le \mu_{\max}(E_2).
\]
Consequently, all inequalities must be equalities: $\mu_{\max}(E_2) = \mu(E) = \mu_{\min}(E_1)$. Now, suppose for contradiction that $E_1$ is not slope-semistable. Then there exists a destabilizing subsheaf $K \subset E_1$ with $\mu_{\min}(K) > \mu(E_1) \ge \mu(E)$. We define a subobject $S := p^*K \otimes \cO_{\Sigma_e}(-C_0)[1] \subset \lambda_1(E)$.
Consider the pullback extension $E' \subset E$ defined by the short exact sequence:
\[
    0 \to \rho_2(E) \to E' \to S \to 0.
\]
The extension class of this sequence lies in $\rExt^1(S, \rho_2(E))$. By the projection formula, we have
\[
    \rExt^1(S, \rho_2(E)) \cong \rhom(K, E_2) \oplus \rhom(K, E_2(-e)).
\]
Since $\mu_{\min}(K) > \mu(E) \ge \mu_{\max}(E_2)$, both Hom-groups vanish. Thus, the extension splits, i.e., $E' \cong \rho_2(E) \oplus S$. This implies that $S$ is a subobject of $E'$ and hence of $E$. Since $\phi(S) = \mu(K) > \mu(E)$, this contradicts the semistability of $E$. Hence, $E_1$ is slope-semistable of slope $\mu(E)$. The statement \eqref{eq:m4-d1-linear} follows from the slope-semistability of $E_1$, which implies \eqref{eq:m4-d-bounds}.

If $E$ is split, we have an isomorphism $E \simeq p^*E_2 \oplus (p^*E_1 \otimes \cO_{\Sigma_e}(-C_0)[1])$ in $\cA_{gl,4}$. Since $E$ is $\sigma_{gl,4}$-semistable, its direct summands must also be $\sigma_{gl,4}$-semistable with the same phase $\phi(E)$.
In particular, $\phi(p^*E_1 \otimes \cO_{\Sigma_e}(-C_0)[1]) = \phi(E)$, which implies $\mu(E_1) = \mu(E)$.
Furthermore, the semistability of $p^*E_1 \otimes \cO_{\Sigma_e}(-C_0)[1]$ in $\cA_{gl,4}$ implies the slope-semistability of $E_1$ on $\bP^1$, as any destabilizing subsheaf of $E_1$ would induce a destabilizing subobject in $\cA_{gl,4}$. Thus, $E_1$ is slope-semistable with slope $\mu(E)$, and the equalities \eqref{eq:m4-d1-linear} and \eqref{eq:m4-d-bounds} hold.
    \end{proof}
\end{lemma}

Now, we fix the norm on $\Lambda_\bR = K_{num} \otimes \bR$ as  
$$
    \|\rch(G)\| := {\rm max} \left \{ |r|,|a|,|b|, |\rch_2| \right \}
$$
for $G \in D^b(\Sigma_e)$, where $c_1(G)$ is denoted by $aC_0+bf$. Using this with the formula \eqref{m4chern}, we obtain
$$
    \| \rch(E) \| = {\rm max} \left \{ |r_2-r_1|, |r_1|, |d_2-d_1|, \left| d_1+\frac{1}{2}e r_1 \right| \right \}
$$
for all $E \in \cA_{gl,4}$. 

\begin{proposition}[{\bf Support property for $m=4$}]\label{prop:m4-support-property}
The glued pre-stability condition $\sigma_{gl,4}=(Z_{gl,4},\cA_{gl,4})$ satisfies the support property. Namely, there exists a constant $C>0$ such that for every $\sigma_{gl,4}$-semistable object $0\ne E\in\cA_{gl,4}$, we have
\[
    \|\rch(E)\|\ \le\ C\,|Z_{gl,4}(E)|.
\]
\end{proposition}

\begin{proof}
    Let $E$ be as in \eqref{eq.exactseqofAgl4} and set $D = d_1+d_2$, $R = r_1+r_2$ defined in \eqref{slopeofZgl4}. If $R=0$, 
    then $r_1=r_2=0$ and $E_1,E_2$ are torsion. In this case $|Z_{gl,4}(E)|=|D|$. Using \eqref{m4chern} with $r_1=r_2=0$, we have $r=0$, $a=0$, $b=d_2-d_1$, and $\rch_2=d_1$. Hence,
    \[
        \|\rch(E)\|=\max\{|d_2-d_1|,\ |d_1|\}\le |d_1|+|d_2|= |D|=|Z_{gl,4}(E)|,
    \]
    and hence the inequality holds with constant $1$.

    Assume now $R>0$. Then $E_1,E_2$ are vector bundles, and Lemma \ref{lem:m4-degree-control} holds. In particular, we have $|d_1|\le |D|$ and $|d_2|\le 2|D|$.
    Using \eqref{m4chern}, the Chern character is given by
    \[
        \rch(E)=(r,aC_0+bf,\rch_2),\quad r=r_2-r_1,\ a=r_1,\ b=d_2-d_1,\ \rch_2=d_1+\frac{e}{2}r_1.
    \]
    We estimate each component in terms of $|Z_{gl,4}(E)|$:
    \begin{itemize}
        \item $|r|=|r_2-r_1|\le R\le |Z_{gl,4}(E)|$.
        \item $|a|=r_1\le R\le |Z_{gl,4}(E)|$.
        \item $|b|=|d_2-d_1|\le |d_1|+|d_2|\le 3|D|\le 3|Z_{gl,4}(E)|$.
        \item $|\rch_2|=\left|d_1+\frac{e}{2}r_1\right| \le |d_1|+\frac{e}{2}|r_1| \le |D|+\frac{e}{2}R \le \Bigl(1+\frac{e}{2}\Bigr)|Z_{gl,4}(E)|$.
    \end{itemize}
    Taking $C := \max\{3, 1+\frac{e}{2}\}$, the inequality holds for all semistable objects.
\end{proof}

\hspace{1mm}

\subsection{Gluing perversity}
In this section, we define gluing perversity and provide the condition for a skyscraper sheaf is $\sigma_{gl}$-stable. In \cite{Uch}, for a pre-stability condition $\sigma \in \cG l_{pre}((Z_1,\cP_1),(Z_2,\cP_2))$ with $\cP_1(0) = p^*\cP_{st}(\phi) \otimes \cO_{\hir}(-C_0)$ and $\cP_2(0) = p^*\cP_{st}$, the {\it gluing perversity} is defined as
$$
    \rper(\sigma) := \phi_1-\phi_2.
$$
In \cite{Uch}, this concept serves roughly three purposes: First, to measure the phase discrepancy of skyscraper sheaves between $\sigma_1$ and $\sigma_2$; second, to prove that $\sigma$ is locally finite stability condition; and finally, to provide a sufficient condition for the gluing condition \eqref{gleq1} to be satisfied. In this paper, however, the latter two properties are established by Proposition \ref{Gluingheart} and Proposition \ref{supportpropertyofgluingstabilitycondition}, respectively. Therefore, we adopt a slightly different definition from \cite{Uch}, redefining perversity for the more limited purpose of discussing the phase discrepancy of skyscraper sheaves. 

\begin{definition} \label{gluingperbersity}
    Let $\sigma \in \cG l(\sigma_1,\sigma_2)$. For a point $x \in \Sigma_e$, the {\it gluing perversity} of $\sigma$ is defined as
    \begin{equation}
         {\rm per}(\sigma) = \phi(\lambda_1(\cO_x)) - \phi(\rho_2(\cO_x))
    \end{equation} 
    where $\lambda_1$ and $\rho_2$ are the adjoint functors defined in Lemma \ref{Uch.Proposition3.3}. 
\end{definition}

\vspace{2mm}
Note that by definition, we have $\rho_2(\cO_x) = \cO_f$ and $\lambda_1(\cO_x) = \cO_f(-C_0)[1]$ for any $x \in \Sigma_e$. Thus the above definition can be rephrased as 
$$
    {\rm per}(\sigma) = \phi(\cO_f(-C_0)[1]) - \phi(\cO_f).
$$

\begin{lemma} \label{phaseoflambda1Oxandrho2Ox}
Let $\sigma \in \cG l(\sigma_1,\sigma_2)$. Then, we have 
\begin{equation}
\phi(\lambda_1(\cO_x)) =  \left \{ \begin{aligned}  
            & 1 & {\rm if  \ \sigma_1 \ is \ standard} \\
	          & \psi(Z_1(\cO(k))+ Z_1(\cO(k-1)[1])	& {\rm if  \ \sigma_1 \ is \ algebraic} \end{aligned}
              \right.
\end{equation}
Similarly, 
\begin{equation}    
\phi(\rho_2(\cO_x)) =  \left \{ \begin{aligned}  
            & 1 & {\rm if  \ \sigma_2 \ is \ standard} \\
	          & \psi(Z_2(\cO(k))+ Z_2(\cO(k-1)[1])	& {\rm if  \ \sigma_2 \ is \ algebraic} \end{aligned}
              \right .
\end{equation}
where $\psi(E)$ is the phase of object $E$ on $\rStab(\bP^1)$.

\begin{proof}
We will first prove the statement for $\phi(\rho_2(\cO_x))$. Assume that $\sigma \in \cG l (\sigma_1,\sigma_2)$. If $\sigma_2$ is standard, then from Proposition \ref{prop.gluingstabilityfunction} and simple calculation of Chern characters, we have $\phi(Z_2(\rho_2(\cO_x))) = 1$ in $\cA_2[j_2] = \cP_2(0,1]$. Assume $\sigma_2$ is algebraic. Pulling back the triangle \eqref{eq.Oka05Lemma3.1(3)} via $p^*$ yields a triangle
\begin{equation} \label{eq.trianglofOf}
	 p^*\cO(k)[j] \to \cO_f[j] \to p^*\cO(k-1)[j+1].
\end{equation}
on $p^*D^b(\bP^1)$. Since $p^*\cO(k-1)[j+1]$ and $p^*\cO(k)[j]$ are objects within the heart $ \cP_2(0,1] = p^*\cA(k)[j]$, the object $\cO_f[j]$ must also be in $\cP_2(0,1]$ by the extension closedness of $p^*\cA(k)[j]$. Therefore we get 
$$
    Z_2(\rho_2(\cO_x)) = Z_2(O_f[j]) = Z_2(p^*O(k)[j]) + Z_2(p^*O(k-1)[j+1]).
$$
The phases of objects in $p^*\sigma_2 \in \rStab(p^*D^b(\bP^1))$ is equals to the phase of objects in $\sigma_2 \in \rStab(\bP^1)$, since the phases is preserved under $p^*$ (\cite{CP} Proposition 2.2(3)). Therefore, this proves the claim for the algebraic case. The statement for $\sigma_1$ is obtained by tensoring the triangle \eqref{eq.Oka05Lemma3.1(3)} with $\cO_{\hir}(-C_0)$, and making a similar argument.
\end{proof}
\end{lemma}

For the glued types $m=1$ and $m=2$, we verify the following property.
\begin{lemma} \label{lem.phaseofzeta0zeta1}
    Let $\sigma=\sigma_{gl,m} \in \cG l(\sigma_1,\sigma_2)$ be a gluing stability condition with $\rper(\sigma) = 0$. Suppose $m \in \{1, 2\}$, so that one of the stability conditions $\sigma_i$ is a quiver stability condition $\sigma_{k,\zeta_0,\zeta_1}$ with $\zeta_0,\zeta_1 \in \bH$. Then the parameters $\zeta_0$ and $\zeta_1$ must lie on the negative real axis $\mathbb{R}_{<0}$.
    \begin{proof}
    We begin with the case $m=1$. By definition, $\sigma_1$ is a standard stability condition, while $\sigma_2$ is of quiver type. By Lemma \ref{phaseoflambda1Oxandrho2Ox}, we identify the phases as $\phi(\lambda_1(\cO_x))=1$ and $\phi(\rho_2(\cO_x)) = \phi(\zeta_0+\zeta_1)$. The vanishing of the gluing perversity, $\rper(\sigma)=0$, imposes the equality:
    $$
        \phi(\lambda_1(\cO_x)) = 1 = \phi(\rho_2(\cO_x)) = \phi(\zeta_0+\zeta_1).
    $$
    This equality forces the sum $\zeta_0+\zeta_1$ to lie on the negative real axis $\mathbb{R}_{<0}$.  Recall that the parameters $\zeta_i$ belong to $\mathbb{H} = \{ r e^{i \pi \phi} \mid r > 0, 0 < \phi \le 1 \}$, which implies $\mathrm{Im}(\zeta_i) \ge 0$ for $i=0,1$. Therefore, the sum $\zeta_0+\zeta_1$ is real if and only if $\mathrm{Im}(\zeta_0) = \mathrm{Im}(\zeta_1) = 0$. Consequently, $\zeta_0$ and $\zeta_1$ must be negative real numbers. For the case $m=2$, the same argument applies by exchanging the roles of $\sigma_1$ and $\sigma_2$.    \end{proof}
\end{lemma}

\begin{definition}
    Let $\sigma \in \cG l(\sigma_1,\sigma_2)$. For $i=1,2$, we define the following phase values:
    \begin{equation}
        {\rm per}_{i}(\sigma) := \left \{    \begin{aligned}
            &1  & {\rm if \ } \sigma_i {\rm \ is \ standard} \\
            &\phi(Z_i(\cO(k))) - \phi(Z_i(\cO(k)[j+1])) &\quad  {\rm if \ } \sigma_i  {\rm \ is \ quiver}
        \end{aligned} \right .
    \end{equation}
\end{definition}

\vspace{1mm}
\begin{remark}
    \begin{enumerate}
        \item The structure sheaf of a fiber $\cO_f$ is the simple object of $p^*\rCoh(C)$(\cite{Uch} Lemma 3.10). That is, its only proper subobjects is the zero object.
        \item If $\sigma_{gl} \in \cG l(\sigma_1,\sigma_2)$, then for every $\phi \in \bR$, we have $\cP_1(\phi) \subset \cP_{gl}(\phi)$ and $\cP_2(\phi) \subset \cP_{gl}(\phi)$ by \cite{CP} Lemma 2.2 (3). \label{rmk1.Gluingstab}
    \end{enumerate} 
\end{remark}

We now discuss the conditions under which a skyscraper sheaf becomes $\sigma_{gl}$-stable. For this purpose, we establish some preliminary results. 

Let $\sigma =(Z,\cA)$ be a stability condition on either $\rStab(p^*D^b(\bP^1))$ or $\rStab(p^*D^b(\bP^1) \otimes \cO_{\hir}(-C_0))$. If $\sigma$ is standard, $\cO_f$ and $\cO_f(-C_0)[1]$ are simple objects in the corresponding heart (\cite{Uch} Lemma 3.10). However, when $\sigma$ is algebraic, i.e., $\cA = \langle p^*\cO(k-1)[1],p^*\cO(k) \rangle$, it is not immediately clear whether $\cO_f$ and $\cO_f(-C_0)[1]$ are simple in $\cA_1$.
More precisely, since $p^*\cO(k-1)[1]$ and $p^*\cO(k)$ are simple, $\cO_f$ is not simple in $\cP_1(\phi)$ if $\phi(p^*\cO(k-1)[1]) = \phi(\cO_f)$ or $\phi(p^*\cO(k)) = \phi(\cO_f)$. But, the converse does not hold in general.

For convenience, let $\cO_{\hir}(n,m)$ denote the line bundle $\cO_{\hir}(nC_0+mf)$ for $n,m \in \bZ$. Note that $p^*\cO(k) \cong \cO_{\hir}(0,k)$, so the heart $\cA$ can be written in this notation as $\cA = \langle \cO_{\hir}(0,k-1)[1], \cO_{\hir}(0,k) \rangle$. 

Recall that the Hirzebruch surface has a strong full exceptional collection 
$$
      \cE := (E_0,E_1,E_2,E_3) = (\cO_{\hir}(0,k-1),\cO_{\hir}(0,k),\cO_{\hir}(1,k-1+e),\cO_{\hir}(1,k+e)).
$$
Using this, we can construct the heart of a bounded t-structure $\cK \subset D^b(\Sigma_e)$ generated by an Ext-exceptional collection as follows:
$$
    \cK = \langle  \cO_{\hir}(0,k-1)[1],\cO_{\hir}(0,k),\cO_{\hir}(1,k-1+e)[-1],\cO_{\hir}(1,k+e)[-2]  \rangle.
$$
A heart $\cA$ is a subcategory of $\cK$. The general objects of $\cK$ are sequences of the form:
$$
    \bC^{m_0} \otimes \cO_{\hir}(0,k-1) \to \bC^{m_1} \otimes \cO_{\hir}(0,k) \to \bC^{m_2}  \otimes\cO_{\hir}(1,k-1+e) \to \bC^{m_3} \otimes \cO_{\hir}(1,k+e),
$$
whose terms are concentrated in degrees $-1,0,1$ and $2$, and are associated with the dimension vector $[m_0,m_1,m_2,m_3]$. In this context, $\cA$ is described as the subcategory of objects in $\cK$ with a specific dimension vector:
$$
    \cA = \{ E \in \cK \mid E \ {\rm has \ a \ dimension \ vector \ }[m_0,m_1,0,0] \}.
$$
A collection $\cE$ induces the equivalence $\Psi_{\cE} : D^b(\Sigma_e) \to D^b({\textrm{\textbf{Mod}}_A^{\textrm{f.g.}}})$, where $A := \textrm{End}(\oplus_{i=0}^3 E_i)$ is the tilting object of $\cE$ \eqref{Bondalequiv}. Moreover, $\Psi_{\cE}$ induces an isomorphism of Grothendieck groups $\psi : K(\hir) \to K(\textrm{\textbf{Mod}}_A^{\textrm{f.g.}})$. We have coordinates on these Grothendieck groups given by the isomorphisms
$$
    K(\hir) \overset{(\rrank(v),\rdeg_{C_0}(v), \rdeg_f(v),\rch_2(v))}{\xrightarrow[]{}} \quad \bZ^{\oplus 4} \quad  \overset{(m_0,m_1,m_2,m_3)}{\xleftarrow[]{}} K(\textrm{\textbf{Mod}}_A^{\textrm{f.g.}}).
$$
The coordinate transformation between these is given by the matrix
\begin{equation}
    \begin{pmatrix}
        m_0 \\
        m_1 \\
        m_2 \\
        m_3
    \end{pmatrix} = \begin{pmatrix}
        -k  & 1  & k+\frac{e}{2} & -1 \\
        1-k & 1  & k-1+\frac{e}{2} & -1 \\
        0   & 0  & -k-\frac{e}{2} &  1 \\
        0   & 0 & 1-k-\frac{e}{2} &  1
    \end{pmatrix} 
    \begin{pmatrix}
        \rrank(v) \\
        \rdeg_{C_0}(v)\\
        \rdeg_f(v) \\
        \rch_2(v)
    \end{pmatrix}, \label{mat11}
\end{equation}
\begin{equation}    
    \begin{pmatrix}
        \rrank(v) \\
        \rdeg_{C_0}(v) \\
        \rdeg_f(v) \\
        \rch_2(v)
    \end{pmatrix} = \begin{pmatrix}
        -1 & 1 & -1 & 1 \\
        1-k & k & 1-k & k \\
        0 & 0 & -1 & 1 \\
        0 & 0 & 1-k-\frac{e}{2} & k+\frac{e}{2}
    \end{pmatrix} 
    \begin{pmatrix}
        m_0 \\
        m_1 \\
        m_2 \\
        m_3
    \end{pmatrix}. \label{mat12}
\end{equation}
Similarly, $\cA \otimes \cO_{\hir}(-C_0)$ is the subcategory of 
$$
    \cK \otimes \cO_{\hir}(-C_0) = \langle \cO_{\hir}(-1,k-1)[1],\cO_{\hir}(-1,k),\cO_{\hir}(0,k-1+e)[-1],\cO_{\hir}(0,k+e)[-2] \rangle
$$
and the coordinate transformation given by 
\begin{equation}
    \begin{pmatrix}
        m_0 \\
        m_1 \\
        m_2 \\
        m_3
    \end{pmatrix} =  \begin{pmatrix}
        0   &  0 &  k+\frac{e}{2} & -1 \\
        0   &  0 &  k-1+\frac{e}{2} & -1 \\
        -k-e  &  1 &  -k-\frac{e}{2} &  1 \\
        1-k-e &  1 &  1-k-\frac{e}{2} &  1
    \end{pmatrix} 
    \begin{pmatrix}
        \rrank(v) \\
        \rdeg_{C_0}(v) \\
        \rdeg_f(v) \\
        \rch_2(v)
    \end{pmatrix}, \label{mat21}
\end{equation}
\begin{equation}    
    \begin{pmatrix}
        \rrank(v) \\
        \rdeg_{C_0}(v) \\
        \rdeg_f(v) \\
        \rch_2(v)
    \end{pmatrix} =  \begin{pmatrix}
        -1 & 1 & -1 & 1 \\
        1-k-e & k+e & 1-k-e & k+e \\
        1 & -1 & 0 & 0 \\
        k-1+\frac{e}{2} & -k-\frac{e}{2} & 0 & 0
    \end{pmatrix} 
    \begin{pmatrix}
        m_0 \\
        m_1 \\
        m_2 \\
        m_3
    \end{pmatrix}. \label{mat22}
\end{equation}

\vspace{1mm}
\begin{lemma} \label{JH-filtofOf}
    Let $\sigma = (Z,\cA)$ be a quiver stability condition where $\cA = p^*\cA(k)[j]$. Then the following hold:
    \begin{enumerate}
        \item $\cO_f[j]$ is included in $\cA$. 
        \item The only nonzero proper subobjects of $\cO_f[j]$ in $\cA$ are exactly $\cO_{\hir}(0,k-1)[j+1]$ and $\cO_{\hir}(0,k)[j]$. 
        \item $\cO_f$ is strictly $\sigma$-semistable if and only if the equality $\phi(\cO_f)=\phi(\cO_{\hir}(0,k))=\phi(\cO_{\hir}(0,k-1)[1])$ holds. In this case, the Jordan-H$\ddot{o}$lder filtration of $\cO_f[j]$ is given by
        $$
            0 \to \cO_{\hir}(0,k)[j] \to \cO_f[j] \to \cO_{\hir}(0,k-1)[j+1] \to 0.
        $$
    \end{enumerate}
    An analogous statement holds when $\cA$ is replaced by $p^*\cA(k)[j] \otimes \cO_{\hir}(-C_0)$ and $\cO_f[j]$ by $\cO_f(-C_0)[j]$.
 \begin{proof}
        For (1), the statement follows from the existence of the triangle \eqref{eq.trianglofOf}. We prove (2). Assume that $F$ be a subobject of $\cO_f$. By a simple calculation of Chern character, we have $v := \rch(\cO_f) = (0,f,0)$, and hence its class has coordinates $(0,1,0,0)$ in $K(\Sigma_e)$.       
        Applying the transformation \eqref{mat11} yields the corresponding dimension vector $ \underline{\rdim}(v) =[1,1,0,0]$. If $E$ is an object in $\cA$ with dimension vector $[m_0,m_1,m_2,m_3]$, then the dimension vector of any subobject, $[m_0',m_1',m_2',m_3']$, satisfies $m_i' \leq m_i$ for all $i$. Therefore, by applying \eqref{mat11}, the nonzero proper subobjects of $\cO_f$ have the dimension vector either $[1,0,0,0]$ or $[0,1,0,0]$, which correspond to $\cO_{\hir}(0,k-1)[1]$ and $\cO_{\hir}(0,k)$ respectively. The general statement for any shift $j$ follows. The analogous claim for $\cO_f(-C_0)[j]$ is proven similarly, using its Chern character $\rch(\cO_f(-C_0)[j]) = (-1)^j(0,f,-1)$ (so its class has coordinates $(-1)^j(0,1,0,-1)$ in $K(\Sigma_e)$) and using matrix \eqref{mat21}.
        For (3), the strict semistability is a direct consequence of (2). Since $\cO_{\hir}(0,k)[j]$ and $\cO_{\hir}(0,k-1)[j+1]$ are simple in $\cA$, the latter statement follows from (2).
    \end{proof}
\end{lemma}

\begin{proposition} \label{PropHN-filtofAgl}    
    Let $\sigma_{gl}=(Z_{gl},\cA_{gl}) \in \cG l (\sigma_1,\sigma_2)$. Then the following statements hold:
\begin{enumerate}
      \item The phase of $\cO_f(-C_0)[1]$ is greater than or equal to the phase of $\cO_f$.
	   \item The skyscraper sheaf $\cO_x$ is $\sigma_{gl}$-stable of phase $\phi(\cO_x)$ if and only if $\rper(\sigma_{gl}) > 0$.
	   \item If $per(\sigma_{gl}) = 0$ then $\cO_x$ is strictly semistable, and its Jordan-Hölder filtration in $\mathcal{A}_{gl}$ is given as follows:
        \begin{enumerate}
            \item If ${\rm per}_{i}(\sigma_{gl}) \neq 0$ for $i=1,2$, then the filtration is
            $$
                 0 \to \cO_f \to \cO_x \to \cO_f(-C_0)[1] \to 0.
            $$                
            \item If $\rper_{i}(\sigma_{gl})=0$ for $i=1,2$, the filtration of $\mathcal{O}_x$ consists of the sequence above, followed by the filtrations of its factors: 
            \begin{itemize}                   
                \item $0 \to \cO_{\hir}(kf) \to \cO_f \to \cO_{\hir}((k-1)f)[1] \to 0$
                \item $0 \to \cO_{\hir}(-C_0+kf)[1] \to \cO_f(-C_0)[1] \to \cO_{\hir}(-C_0+(k-1)f)[2] \to 0$
            \end{itemize}
        \end{enumerate}
\end{enumerate}
\begin{proof}
    Let $\phi(E)$ denote the phase of an object $E \in \cA_{gl}$ with respect to $Z_{gl}$. For the assertion (1). For any object $E \in \cA_{gl}$, there is an exact sequence in $\cA_{gl}$ of the form:
    \begin{equation}        
        0 \to \rho_2(E) \to E \to \lambda_1(E) \to 0.  \label{lem3.17eq1}
    \end{equation}
    This implies the phase inequality $\phi(\rho_2(E)) \leq \phi(E) \leq \phi(\lambda_1(E))$. Applying this to $E=\cO_x$ and recalling that $\rho_2(\cO_x) = \cO_f$ and $\lambda_1(\cO_x) = \cO_f(-C_0)[1]$, we immediately prove  (1). To prove the necessity of (2), assume $\cO_x$ is $\sigma_{gl}$-stable. From (1), we knew $\rper(\sigma_{gl}) \geq 0$. If we had ${\rm per}(\sigma_{gl}) = 0$, the phases would be equal: $\phi(\cO_x) = \phi(\cO_f) = \phi(\cO_f(-C_0)[1])$. The exact sequence 
    \begin{equation} \label{exactseqofOx}        
        0 \to \cO_f \to \cO_x \to \cO_f(-C_0)[1] \to 0
    \end{equation}
    shows that $\cO_x$ has a proper subobject $\cO_f$ with the same phase. This contradicts the definition of stability. Therefore, we must have $\rper(\sigma_{gl}) > 0$. For the sufficiency of (2), assume ${\rm per}(\sigma_{gl}) > 0$. Suppose, for the sake of contradiction, that $\cO_{x}$ is not stable in $\cP(\phi(\cO_x))$. Then there must exist a simple subobject $F \subset \cO_x$ such that $\phi(F) = \phi(\cO_x)$. We have an exact sequence in $\cA_{gl}$ of the form:
    \begin{equation}
        \begin{CD}
            0 @>>> F @>>> \cO_x @>>> \cO_x / F @>>> 0 \\
            @. @. @VVV @. @. \\
            0 @>>> \cO_f @>>> \cO_x @>\gamma>> \cO_f(-C_0)[1] @>>> 0.
        \end{CD}
    \end{equation}
    Now, consider the composition of maps $\alpha : F \hookrightarrow \cO_x \overset{\gamma}{\twoheadrightarrow} \cO_f(-C_0)[1]$. 
    We have two cases: 
    \begin{itemize}
        \item Case $\alpha=0$; In this case, $F$ is contained in $\textrm{Ker}(\gamma)$, which means $F$ is a subobject of $\cO_f$. This implies $\phi(F) \leq \phi(\cO_f)$. Since we assumed $\rper(\sigma_{gl})>0$, we have the strictly inequality $\phi(\cO_f) < \phi(\cO_x)$. Combining these gives $\phi(F) < \phi(\cO_x)$, which contradicts our assumption that $\phi(F)=\phi(\cO_x)$.

        \item Case $\alpha \neq 0$; Clearly, $\alpha$ is surjective. Moreover, $\alpha$ must be injective since $F$ is simple. Thus $\alpha$ is an isomorphism, which implies $\phi(F) = \phi(\cO_f(-C_0)[1])$. However, our assumption $\rper(\sigma_{gl}) > 0$ means $\phi(\cO_x) < \phi(\cO_f(-C_0)[1])$. This leads to $\phi(\cO_x) < \phi(F)$, which contradicts the inequality $\phi(F) = \phi(\cO_x)$.
    \end{itemize}
    Since both cases lead to a contradiction, our assumption that $\cO_x$ is not stable must be false. The assertion (3) follows directly from Lemma \ref{JH-filtofOf}, parts (2) and (3).
\end{proof}
\end{proposition}

In general, for a smooth projective surface $S$ over $\bC$, if $\cU$ is the set of geometric stability condition of $\rStab_\Lambda(S)$ (i.e. the skyscraper sheaves are stable of the same phase), then it is a well-known fact that $\cU$ is an open subset of $\rStab_\Lambda(S)$ (\cite{Br08}). Consequentry, the boundary $\partial \cU := \overline{\cU} \setminus \cU$ coincides with the set of stability conditions where all skyscraper sheaves are strictly $\sigma$-semistable. From \ref{PropHN-filtofAgl}, for gluing stability conditions, this strictly $\sigma$-semistability is equivalent to the condition $\rper(\sigma) = 0$.

This suggests the following definition
\begin{definition} \label{dfn.wallofgluingperversity0}
    We define the {\it wall} $\cW_0$ as the subset of $\rStab_\Lambda(\Sigma_e)$ consists of all gluing stability conditions with zero gluing perversity.
\end{definition}

The condition $\rper(\sigma) = 0$ is equivalent to the phase-matching equation
\begin{equation}    
    \rRe Z_{gl}(\lambda_1(\cO_x)) \rIm Z_{gl}(\rho_2(\cO_x)) - \rIm Z_{gl}(\lambda_1(\cO_x)) \rRe Z_{gl}(\rho_2(\cO_x))=0. \label{equationofwall}
\end{equation}
This is a single, non-trivial real equation on the eight-dimensional real manifold $\mathrm{Stab}_\Lambda(\Sigma_e)$. Therefore, its solution set $\cW_0$ forms a real codimension-one submanifold (a 7-dimensional real submanifold) of $\mathrm{Stab}_\Lambda(\Sigma_e)$, which lies on the boundary $\partial \mathcal{U}$.

We denote by $\mathcal{S}_m$ the set of gluing stability conditions of type $m$. We are interested in the trace of this global wall on each specific type:
\begin{definition} \label{localwall}
    For each type $m=1,2,3,4$, we define the subset $\cW_{0,m}$ as the intersection the global wall $\cW_0$ and the subset $\cS_m$: $\cW_0 \cap \cS_m$.
\end{definition}

\subsection{Destabilizing wall of skyscraper sheaves}
In this section, we prove that the set $\cW_{0}$ defined in \eqref{dfn.wallofgluingperversity0}, is a destabilizing wall of skyscraper sheaves. We identify $\rhom_\bZ(K_{num}(\Sigma_e),\bC)$ with $K_{num}(\Sigma_e) \otimes_\bZ \bC$. 

The main purpose of \cite{Uch} is to find the {\it destabilizing wall of skyscraper sheaves} in $\rStab_\Lambda(S)$.
\begin{definition}[\cite{Uch} Definition 2.8] \label{dfn.destabilizingwallofskyscrapersheaves}
    A set $\cW \subset \rStab_\Lambda(S)$ is called a {\it destabilizing wall of skyscraper sheaves} if it satisfies the following properties:
    \begin{enumerate}
        \item $\cW$ is the real codimension one submanifold of $\rStab_\Lambda(S)$. \label{cnd.destabilizingwall1}
        \item For any $\sigma=(Z,\cP) \in \cW$ and any point $x \in S$, there exists an exact sequence $0 \to T \to \cO_x \to F \to 0$ of $\sigma$-semistable objects in $\cP(\phi)$ for some $\phi \in \bR$. \label{cnd.destabilizingwall2}
        \item For any $\sigma=(Z,\cP) \in \cW$, there exists an $\epsilon_0 > 0$ such that if $0 < \epsilon < \epsilon_0$ and $W : \Lambda \to \bC$ satisfying
        $$
            |W(E) - Z(E)| < {\rm sin}(\pi \epsilon) |Z(E)|
        $$
        for all $E$ which $\sigma$-semistable, then there is a geometric stability condition $(W,\cQ)$ with $d(\cP,\cQ) < \epsilon$. \label{cnd.destabilizingwall3}
    \end{enumerate}
\end{definition}
Property (3) requires that a destabilizing wall intersects the boundary of the set of geometric stability conditions.

We have already seen that for any stability condition in $\cW_{0}$, the skyscraper sheaves $\cO_x$ has the Jordan-H$\ddot{o}$lder filtration given in \eqref{PropHN-filtofAgl}. This shows that $\cW_{0}$ has condition (2) of Definition \ref{dfn.destabilizingwallofskyscrapersheaves}. Furthermore, it satisfies condition (1), since the equation \eqref{equationofwall}. 

We will now show that it also satisfies the condition (3). The strategy for proving (3) is to replace the assumption $\phi(\cO_x)=1$ in the proof of \cite{Uch} Lemma 4.2 by a general value.
\begin{proposition}
    For $ \sigma_0 := (Z_0,\cP_0) \in \cW_{0}$, there exists an $\epsilon_0 > 0$ with the following property: for any $0 < \epsilon < \epsilon_0$ and any group homomorphism $W : \Lambda \to \bC$ satisfying 
    \begin{itemize}
        \item $\phi(W(\cO_f(-C_0))[1]) > \phi(W(\cO_f))$
        \item $|W(E) - Z_0(E)| < {\rm sin} (\pi \epsilon) |W(E)|$ for any $E \in D^b(S)$ semistable in $\sigma_0$,
    \end{itemize}
    there exists a unique locally finite geometric stability condition $\sigma = (W,\cQ)$ with $d(\cP_0,\cQ) < \epsilon$.
    \begin{proof}
        The approach is similar to that in \cite{Uch} Lemma 4.2. Let $\sigma_0 = (Z_0, \mathcal{P}_0) \in \mathcal{W}_{0}$. By Proposition \ref{thm.Br07Theorem7.1}, there exists a locally finite stability condition $\sigma = (W,\cQ)$. Let $r = \phi_0(\mathcal{O}_x)$ be the phase of the skyscraper sheaf. We may set the enveloping subcategory to be $\mathcal{P}_0((r-2\epsilon, r+2\epsilon))$. Furthermore, by rotating $\sigma_0$ using the $\mathbb{C}^*$-action on $\mathcal{W}_{0}$, we can assume $r=1$. By construction, the image of $Z_0$ is a discrete subset of $\mathbb{C}$. Therefore, the proof is completed by applying an argument similar to that in \cite{Uch} Lemma 4.2.
    \end{proof}
\end{proposition}

\section{Relation to divisorial stability conditions} \label{Relationtodivisorialstability}
In this section, we consider divisorial stability conditions and discuss the destabilizing wall within this space. We begin by fixing our notation:
\begin{equation}
    B \in \rns(\Sigma_e)_\bR, \quad \omega \in \ramp(\Sigma_e)_\bR. \label{eq.divisornotation}
\end{equation}
Recall that for a smooth projective surface $S$ over $\bC$, a stability condition
\begin{equation} \label{divisorialstabilitycondition} 
    \sigma_{\omega,B}=(Z_{\omega,B},\cA_{\omega,B}) 
\end{equation} is called a {\it divisorial stability condition} if its stability function $Z_{\omega,B}$ has the form 
$$
    Z_{\omega,B}(E) = - \int_S \rexp(B+ i \omega)\rch(E)
$$
for some $B \in \rns(S)$ and $\omega \in \ramp(S)$, and its heart $\cA_{\omega,B}$ is constructed as follows:  The $\mu_\omega$-slope of a sheaf $E \in \rCoh(S)$ is defined by
$$
    \mu_{\omega}(E) = \frac{c_1(E). \omega}{\rrank(E)}.
$$
For any $B \in \rns(S)$ and $\omega \in \ramp(S)$, there exists a unique torsion pair $(\cT_{\omega,B}, \cF_{\omega,B})$ on the category $\rCoh(S)$, where $\cT_{\omega,B}$ consists of sheaves whose torsion-free parts have $\mu_\omega$-semistable Harder-Narasimhan factors with slope $\mu_{\omega,B} > B . \omega$, and $\cF_{\omega,B}$ consists of torsion free sheaves on $S$ all of whose $\mu_\omega$-semistable Harder Narasimhan factors have slope $\mu_\omega \leq B.\omega$.  Moreover, the extension closure $\cA_{\omega,B} := \langle \cF_{\omega,B}[1],\cT_{\omega,B} \rangle$ is the heart of a bounded t-structure on $D^b(S)$. For a smooth projective surface $S$ over $\bC$, there is a natural continuous embedding
\begin{equation} \label{divisorialembedding}    
     \ramp(S)_\bR \times \rns(S)_\bR \hookrightarrow \rStab_\Lambda(S)
\end{equation}
given by the correspondence $(\omega,B) \mapsto \sigma_{\omega,B}$. We denote the image of this map by $S_{div}$.

As shown in \cite{Br07} \S 6, $\rStab_\Lambda(S)$ has the natural structure of  a complex manifold of dimension $\rrank(\Lambda)$. Its complex structure is induced by the local homeomorphism
\begin{equation}
    \mathcal{\pi} : \rStab_\Lambda(S) \to \rhom_\bZ(\Lambda,\bC) \label{localhomeo}
\end{equation}
defined by $(Z,\cA) \mapsto Z$. Furthermore, $\Lambda = K_{num}(S)$ possesses a symmetric bilinear form $\langle -,-\rangle_M$, known as the {\it Mukai pairing}, defined for $\rch(E) = (r,c_1,\rch_2)$ and $\rch(E')=(r',c_1',\rch_2')$ as
\begin{equation}    
    \langle E,E'\rangle_M = c_1.c_1' - r'\rch_2 - r\rch_2'. \label{MuakiPairing}
\end{equation}
Recall that for any numerical stability condition $\sigma = (Z,A) \in \rStab_{\Lambda}(S)$, there is a vector $\pi(\sigma) \in \Lambda \otimes_\bZ \bC$ such that 
\begin{equation} \label{numericalstab}
    Z(E) = \langle \pi(\sigma),\rch(E) \rangle_M.
\end{equation}
 (\cite{Oh} \S 3, \cite{Uch} \S2). 

It is a well-known fact that any stability condition $\sigma \in S_{div}$ is geometric. In \cite{Oh}, the author provides the following criterion for a stability condition to be divisorial. 
\begin{proposition} \label{ohkawa.proposition3.6}
    $\sigma \in \rStab(S)$ is dvisorial if and only if 
    \begin{enumerate}
    \item $\sigma$ is geometric, and
    \item There exists $M \in GL^+(2,\bR)$ and $B , \omega\in NS(S)_\bR$, where $\omega$ is ample, such that 
$$
    \pi(\sigma) M = \rexp(B+i \omega)
$$
where $\pi$ is the local homeomorphism given in \eqref{localhomeo}.    
\end{enumerate}
\end{proposition}

In the case of Hirzebruch surface, $S_{div}$ can be identified with the following positive cone via the embedding \eqref{divisorialembedding}:
\begin{equation} \label{Sdivtocone}    
    S_{div} = \{ (x,y,z,w) \in \bR^4 \mid z > 0, \ w > ze \}
\end{equation}
where $e = \rdeg \ \cE$, $B = xC_0+yf$, and $\omega = zC_0+wf$ is ample. This set is open in its closure $\overline{S_{div}}$, which is the locus within the nef cone defined by $z \geq 0$ and $w \geq ze$. We denote the boundaries by $\partial_z = \{ (x,y,z,w) \mid z =0, \ w \geq 0 \}$ and $\partial_w = \{ (x,y,z,w) \mid  z \geq0, \ w = ze \}$. The boundary of the closure can then be written as
\begin{equation} 
    \partial \overline{S_{div}} = \partial_z \cup \partial_w.\label{boundaryofwall}
\end{equation}
Recall that $\cW_{0,m}$ is defined as the set of gluing stability conditions of glued type $m$ with gluing perversity zero \eqref{localwall}. We will perform an explicit calculation of $\partial \overline{S_{div}} \cap \cW_{0,m}$, following the method in \cite{Uch} Theorem 4.4.

In \cite{Uch}, the author calculates the image of the local homeomorphism $\pi : \rStab(S)_\Lambda \to \rhom(\Lambda,\bC)$ for a gluing stability condition $\sigma$ in the case where $g(C)>0$.
\begin{lemma}[\cite{Uch} Proposition 3.5] \label{Uch.prop3.5}
    Let $M = \begin{pmatrix}
        a & b \\ c & d
    \end{pmatrix}\in GL^+(2,\bR)$. Suppose that $\sigma_1$ be the stability condition on $p^*D^b(C) \otimes \cO_{\hir}(-C_0)$ and $\sigma_2$ be the standard stability condition on $p^*D^b(C)$. Then gluing stability condition $\sigma_{gl} = (Z_{gl},\cP_{gl})$ glued from $\sigma_1M$ and $\sigma_2$ satisfies
    \begin{equation*}
        \pi(\sigma) = \left(1-a-ic, -C_0+ \left(-\frac{1}{2} e(a+1) -b + i \left(-\frac{1}{2}ce+1-d \right) \right)f, -i \right)
    \end{equation*}
\end{lemma}
This calculation corresponds to our case for glued type $m=4$. However, for the cases where $m=1$, $2$, and $3$, which involve at least one quiver stability condition, the definition of the central charge differs. Therefore, the computation of $\pi(\sigma)$ must be adjusted accordingly.

Recall that if $\cA_{gl}$ is glued from $\cA_1[j_1]$ and $\cA_{2}[j_2]$, we fix the shift degrees as $j_1 = 1$ and $j_2 = 0$. As explained in the previous section, the stability function $Z_{gl,m}$ can be written as 
\begin{equation} \label{eq.Zgl1}
Z_{gl,1} = (-1)(- d_{\lambda_1}(E)  + ir_{\lambda_1}(E))  +(d_{\rho_2}-k r_{\rho_2} )\zeta_0 + (d_{\rho_2} + (1-k) r_{\rho_2})\zeta_1
\end{equation}
\begin{equation} \label{eq.Zgl2}
Z_{gl,2} = -d_{\rho_2}(E) + ir_{\rho_2}(E)+ (-1)((d_{\lambda_1}-k r_{\lambda_1} )\zeta_0 + (d_{\lambda_1} + (1-k) r_{\lambda_1})\zeta_1)
\end{equation}
\begin{align} \label{eq.Zgl3}
        Z_{gl,3} = (-1)(d_{\lambda_1}-k r_{\lambda_1} )\zeta_0 &+ (d_{\lambda_1} + (1-k) r_{\lambda_1})\zeta_1) \\ 
        &+ (d_{\rho_2}-k r_{\rho_2} )\zeta'_0 + (d_{\rho_2} + (1-k) r_{\rho_2})\zeta'_1.
\end{align} More precisely, 
\begin{eqnarray*} 
    & & d_{\lambda_1} - k r_{\lambda_1} =  -\rch_2 + (\frac{1}{2}e+k) c_1.f\\
    & &d_{\rho_2} -k r_{\rho_2} = \rch_2 + c_1.C_0 + (\frac{1}{2}e-k) c_1.f - kr. \\
\end{eqnarray*}
We now calculate the image of the local homeomorphism $\pi(\sigma)$ for these stability conditions. As explained in \eqref{numericalstab}, any stability function $Z$ associated with a stability condition in $\mathrm{Stab}_\Lambda(\Sigma_e)$ can be written as $Z(E) = \langle \pi(\sigma)(E),\rch(E) \rangle_M$ (\cite{Uch} \S 2). Therefore, computing $\pi(\sigma)$ reduces to solving the equation:
\begin{equation}
    Z_{gl,m} = \langle \pi(\sigma),\rch\rangle_M. \label{eq.Zglm=Mukaipairing}
\end{equation}
Before proceeding with the calculation, let us establish some general properties that hold for any divisorial stability condition. Let $\pi(\sigma) = (\xi_0,\xi_1,\xi_2,\xi_3)$ be a solution to the above equation. By the proposition \ref{ohkawa.proposition3.6}, $\sigma$ is divisorial if and only if 
$$
    \pi(\sigma)M = e^{B+ i \omega}
$$
for some $B = xC_0+yf$, $\omega = zC_0 + fw$ with $\omega$ ample, and $M= \begin{pmatrix}
    a & b \\
    c & d
\end{pmatrix} \in GL^+(2,\bR)$. Viewing both sides as vectors in the  $8$-dimensional real vector space $\mnH^0(\Sigma_e,\bC) \oplus \rns(\Sigma_e)_\bC \oplus \mnH^4(\hir,\bC)$, the left-hand side can be written as 
\begin{equation}
    \pi(\sigma)M = \left( 
       \begin{pmatrix} a\rRe \xi_i + b \rIm \xi_i \\ c \rRe \xi_i + d \rIm \xi_i\end{pmatrix} \right)_{i=0,1,2,3}.
\end{equation}
where we have decomposed $\pi(\sigma)$ into components corresponding to $[1], [C_0], [f], [\textrm{pt}]$. Similary, the right-hand side is 
\begin{equation}
    e^{B+i \omega} = \left (\begin{pmatrix} 1 \\ 0 \end{pmatrix},\begin{pmatrix} x \\ z \end{pmatrix} C_0,\begin{pmatrix} y \\ w \end{pmatrix}f, \begin{pmatrix} \frac{1}{2}((z^2-x^2)e + 2(xy-zw)) \\ yz+xw-xze \end{pmatrix} \textrm{pt} \right).
\end{equation}

Forcusing on the imaginaly part of the first term(corresponding to $\mnH^0(\Sigma_e,\bC)$-components), we find $c \rRe \xi_0 + d \rIm \xi_0 = 0$. This implies that a vector $\begin{pmatrix}
    c \\ 
    d
\end{pmatrix}$ is orthogonal to $\begin{pmatrix}
    \rRe \xi_0 \\ 
    \rIm \xi_0
\end{pmatrix}$. 
Thus, there exists a non-zero proportiona
lity constant $t \in \bR$ such that $c = -t \rIm  \xi_0$ and $d = t \rRe \xi_0$. Substituiting these into the expression for $z = c \rRe \xi_1 + d \rIm \xi_1$ and $w = c \rRe \xi_2 + d \rIm \xi_2$ yields 
\begin{equation} \label{eq.z}
    z = t  \cdot \rdet \begin{pmatrix}
        \rRe \xi_0 & \rRe \xi_1 \\
        \rIm \xi_0 & \rIm \xi_1
    \end{pmatrix},
\end{equation}
\begin{equation} \label{eq.w}
    w = t  \cdot \rdet \begin{pmatrix}
        \rRe \xi_0 & \rRe \xi_2 \\
        \rIm \xi_0 & \rIm \xi_2
    \end{pmatrix}.
\end{equation}
For convinience, we denote $\rdet(\xi_{i,j}) =  \rdet \begin{pmatrix}
        \rRe \xi_i & \rRe \xi_j \\
        \rIm \xi_i & \rIm \xi_j
    \end{pmatrix}$. The ampleness of $\omega$ requiers $z >0$ and $w > ze$, which leads to the following conclusion.
\begin{lemma} \label{lem.condw>ze}
    Let $\sigma$ be an arbitary stability condition on $\rStab_\Lambda(\Sigma_e)$. If $\sigma$ is divisorial, then the following inequality hold:
    $$
        \frac{\rdet (\xi_{0,2}) - e \cdot \rdet (\xi_{0,1})}{\rdet(\xi_{0,1})} > 0
    $$    
    \begin{proof}
        The inequality is obtained by substituting the expressions for $z$ and $w$ from \eqref{eq.z} and \eqref{eq.w} into the ampleness condition $z > 0$ and $w > ze$.
    \end{proof}
\end{lemma}

We will compute an equation \eqref{eq.Zglm=Mukaipairing} and determine the boundary $\partial \overline{S_{div}} \cap \cW_{0,m}$ sepalately for each case of $m=1,2,3$.

\subsection{The case of $m=1$}
We now compute the image of $\pi(\sigma)$ for the case $m=1$. Recall that a stablity condition $\sigma_{gl,1} = (Z_{gl,1},\cA_{gl,1})$ of this type is constructed by: 
$$
    Z_{gl,1} = (-1)(- d_{\lambda_1}(E)  + ir_{\lambda_1}(E))  +(d_{\rho_2}-k r_{\rho_2} )\zeta_0 + (d_{\rho_2} + (1-k) r_{\rho_2})\zeta_1
$$
$$
    \cA_{gl,1} = \cG l(p^* \cA(k)[1] \otimes \cO_{\hir}(-C_0), p^*\rCoh(\bP^1)).
$$
where $\zeta_j = x_j+iy_j \in \mathbb{H}$ for $j=0,1$.

\begin{proposition} \label{cofficientsofpisigma1} 
    Let $\sigma = \sigma_{gl,1}$. Then the corresponding element $\pi(\sigma) \in \rhom_\bZ(\Lambda,\bC)$ has comopnents $(\xi_0,\xi_1,\xi_2, \xi_3)$ given by:   
    \begin{eqnarray*} &\xi_0& = 1-(x_0+x_1)-i(y_0+y_1), \\
        &\xi_{1}& = x_0+x_1 + i(y_0+y_1), \\
        &\xi_{2}& = \frac{1}{2}e+(\frac{1}{2}e-k)x_0+(\frac{1}{2}e-k+1)x_1+i(1+(\frac{1}{2}e-k)y_0+(\frac{1}{2}e-k+1)y_1), \\
        &\xi_3& = kx_0-(1-k)x_1+i(ky_0-(1-k)y_1). 
    \end{eqnarray*}
    \begin{proof}
        Our goal is to solve the equation 
        $$
            Z_{gl,1} = \langle \pi(\sigma),\rch(E)\rangle_M
        $$
        for $\pi(\sigma)$. Let $\rch(E) = (r,c_1,\rch_2)$ and let components of $\pi(\sigma)$ be $(\xi_0,\xi_1C_0+\xi_2f,\xi_3)$. The right-hand side of the equation, the Mukai pairing, can be written as
        $$
             \langle \pi(\sigma),\rch(E)\rangle_M = c_1(\xi_1C_0+\xi_2f) - r\xi_3 - \rch_2\xi_0.
        $$
        The left-hand side,$Z_{gl,1}$, can be expanded using the formulas from \eqref{eq.Zgl1}:
        $$
            Z_{gl,1} = (-1)(- d_{\lambda_1} + ir _{\lambda_1} )+ (d_{\rho_2} - kr_{\rho_2})\zeta_0 + (d_{\rho_2} - (1-k) r_{\rho_2})\zeta_1.
        $$ 
        By substituiting these formulas and rearranging the terms with respect to the components of Chern characters $(r,c_1,\rch_2)$, we get:
       \begin{eqnarray*}
             Z_{gl,1}&=&((\zeta_0+\zeta_1)C_0+(\frac{1}{2}e+(\frac{1}{2}e-k)\zeta_0+(\frac{1}{2}e-k+1)\zeta_1+i)f)c_1 \\
            &-& \left(k \zeta_0-(1-k)\zeta_1 \right)r \\
            &-& (1-\zeta_0-\zeta_1)\rch_2.
        \end{eqnarray*}
        By compairing the cofficients $r$, $c_1$ and $\rch_2$ on both sides, we obtaine the expressions for $\xi_0,\xi_1,\xi_2$ and $\xi_3$ as follows:
         \begin{eqnarray*} &\xi_0& = 1-(x_0+x_1)-i(y_0+y_1), \\
        &\xi_{1}& = x_0+x_1 + i(y_0+y_1), \\
        &\xi_{2}& = \frac{1}{2}e+(\frac{1}{2}e-k)x_0+(\frac{1}{2}e-k+1)x_1+i(1+(\frac{1}{2}e-k)y_0+(\frac{1}{2}e-k+1)y_1), \\
        &\xi_3& = kx_0-(1-k)x_1+i(ky_0-(1-k)y_1). 
    \end{eqnarray*}

    \end{proof}
\end{proposition}

\begin{thm} \label{thm.boundaryofW01}
    The intersection $\partial \overline{S_{div}} \cap \cW_{0,1}$ is presisely $\cW_{0,1}$. In particular, $\mathcal{W}_{0,1}$ is contained entirely within the boundary component $\partial_{z}$ and does not intersect the vertex of $\partial \overline{S_{div}}$. 
    \begin{proof}
        By definition of $\mathcal{W}_{0,1}$, any stability condition $\sigma \in \mathcal{W}_{0,1}$ satisfies $\mathrm{per}(\sigma)=0$. Lemma \ref{lem.phaseofzeta0zeta1} then implies that
        \begin{center}
            $r := x_0+x_1 \in \mathbb{R}_{< 0}$ and $y_0=y_1 = 0$ \label{eq.thmboundaryofW01}
        \end{center} 
       Substituiting these into the formulas from Proposition \ref{cofficientsofpisigma1}, we find the first two components of $\pi(\sigma)$ to be $\xi_0=2$ and $\xi_1 = r$. Since both $\xi_0$ and $\xi_1$ are real, we have $\rIm \xi_0=\rIm \xi_1=0$. The condition for a point to lie on the boundary component $\partial_z$ is $z=0$, which, from \eqref{eq.z}, is equivalent to $\rdet(\xi_{0,1})=0$. We can verify this:
       $\rdet(\xi_{0,1}) = \begin{pmatrix}
        \rRe \xi_0 & \rRe \xi_1 \\
        \rIm \xi_0 & \rIm \xi_1
        \end{pmatrix} = \begin{pmatrix}
        2 & r \\
        0 & 0
        \end{pmatrix} = 0.$
        Thus, any $\sigma \in \cW_{0,1}$ lies on $\partial_z$. The vertex of $\partial \overline{S_{div}}$ is locus where both $z=0$ and $w=ze=0$ hold. From equation \eqref{eq.w}, the condition $w=0$ is equivalent to $\rdet(\xi_{0,2})=0$. Using $y_0=y_1=0$, the imagenary part of $\xi_2$ is 
        $$
            \rIm \xi_2=1+(\frac{1}{2}e-k)y_0+(\frac{1}{2}e-k+1)y_1 = 1.
        $$
        Therefore, 
        $$
            \mathrm{det}(\xi_{0,2}) = \mathrm{det} \begin{pmatrix} \mathrm{Re}\,\xi_0 & \mathrm{Re}\,\xi_2 \\ \mathrm{Im}\,\xi_0 & \mathrm{Im}\,\xi_2 \end{pmatrix} = \mathrm{det} \begin{pmatrix} 2 & \mathrm{Re}\,\xi_2 \\ 0 & 1 \end{pmatrix} = 2 \neq 0.
        $$
        Since $w \neq 0$, the wall $\mathcal{W}_{0,1}$ does not intersect the vertex $\partial_z \cap \partial_w$.
    \end{proof}
\end{thm}

\begin{figure}[htbp]
  \centering 
  \begin{tikzpicture}
    \draw[->, thick] (-1, 0) -- (5, 0) node[below left] {$z$};
    \draw[->, thick] (0, -1) -- (0, 5) node[below left] {$w$};
    \node[below left] at (0,0) {$O$};

    \fill[blue!20] (0,0) -- (0,4) -- (4,4) -- cycle;
    \node[below left] at (1.7,3) {$S_{div}$};
    \draw[blue, very thick] (0,4) -- (0,0) node[midway,right] {$\partial_z$};
    \draw[blue, very thick] (0,0) -- (4,4) node[midway,right] {$\partial_w$};
    \draw[red, very thick] (0,3) -- (0,1) node[midway,left] {$\cW_{0,m}$};

    \node[above, rotate=45] at (2.5, 2.5) {$w=ze$};
    \node[above right] at (0, 4) {$z=0$};
  \end{tikzpicture}

  \caption{The region $\overline{S_{div}}$ for glued type $m=1,2$ at a fixed point $(x,y) \in \bR^2$, and the wall $\cW_{0,m}$ contained within the boundary $\partial_z$.}
  \label{fig:zw_region12} 
\end{figure}
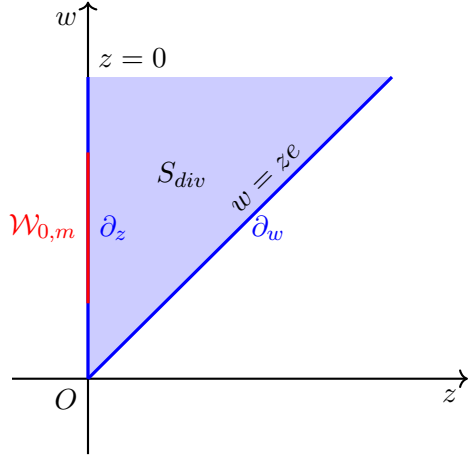

\subsection{The case of $m=2$}
Next, we compute the components of $\pi(\sigma)$ for the case $m=2$. A stability condition $\sigma_{gl,2}=(Z_{gl,2},\mathcal{A}_{gl,2})$ of this type is constructed by
$$
    Z_{gl,2} = -d_{\rho_2} + i r_{\rho_2} + (-1)((d_{\lambda_1}-k r_{\lambda_1} )\zeta_0 + (d_{\lambda_1} + (1-k) r_{\lambda_1})\zeta_1)
$$
and
$$
    \cA_{gl,2} = \cG l(p^* \rCoh(\bP^1)[1] \otimes \cO_{\hir}(-C_0), p^*\cA(k)).
$$
where $\zeta_j = x_j + iy_j \in \bH$ for $j=0,1$.

\begin{proposition} \label{cofficientsofpisigma2} 
    Let $\sigma = \sigma_{gl,2}$. Then the corresponding element $\pi(\sigma) \in \rhom_\bZ(\Lambda,\bC)$ has comopnents $(\xi_0,\xi_1,\xi_2, \xi_3)$ given by:   
    \begin{eqnarray*} &\xi_0& = 1-(x_0+x_1)-i(y_0+y_1), \\
        &\xi_{1}& = -1, \\
        &\xi_{2}& = -\frac{1}{2}e-((\frac{1}{2}e+k)x_0+(\frac{1}{2}e+k-1)x_1)+i(1-((\frac{1}{2}e+k)y_0+(\frac{1}{2}e+k-1)y_1)), \\
        &\xi_3& = -i. 
    \end{eqnarray*}

    \begin{proof}
        Following the same method as in the proof of Proposition \ref{cofficientsofpisigma1}. We solve the equation $Z_{gl,2} = \langle \pi(\sigma),\mathrm{ch}(E) \rangle_{M}$ for $\pi(\sigma) = (\xi_0, \xi_1 C_0 + \xi_2 f, \xi_3)$. The left-hand side, $Z_{gl,2}$, can be expanded using Proposition \eqref{eq.Zgl2}. Rearranging the terms with respect to the components of $\mathrm{ch}(E) = (r, c_1, \mathrm{ch}_2)$ yields:
        \begin{align*} 
            Z_{gl,2} =  & ( -C_{0}+(-\frac{1}{2}e-((\frac{1}{2}e+k)x_{0}+(\frac{1}{2}e+k-1)x_{1}) \\
                        & \qquad \qquad + i(1-((\frac{1}{2}e+k)y_{0}+(\frac{1}{2}e+k-1)y_{1})))f )c_1 \\ 
                        & - (-i)r \\ 
                        & - (1-(x_{0}+x_{1})-i(y_{0}+y_{1}))\mathrm{ch}_2. 
        \end{align*}
        By comparing the coefficients of $r$, $c_1$ and $\mathrm{ch}_2$ with those of $\langle \pi(\sigma),\mathrm{ch}(E) \rangle_{M} = c_1 \cdot (\xi_1 C_0 + \xi_2 f) - r\xi_3 - \mathrm{ch}_2\xi_0$, we obtain the expressions for $\xi_0, \xi_1, \xi_2$, and $\xi_3$.
        \begin{eqnarray*} &\xi_0& = 1-(x_0+x_1)-i(y_0+y_1), \\
        &\xi_{1}& = -1, \\
        &\xi_{2}& = -\frac{1}{2}e-((\frac{1}{2}e+k)x_0+(\frac{1}{2}e+k-1)x_1)+i(1-((\frac{1}{2}e+k)y_0+(\frac{1}{2}e+k-1)y_1)), \\
        &\xi_3& = -i. 
        \end{eqnarray*}

    \end{proof}
\end{proposition}

\begin{thm} \label{thm.boundaryofW02}
    The boundary $\partial \overline{S_{div}} \cap \cW_{0,2}$ is presisely $\cW_{0,2}$. In particular, $\cW_{0,2}$ exactly contained in $z$-boundary $\partial_z$ \eqref{boundaryofwall}, and it does not intersects the vertex of $\partial \overline{S_{div}}$.
    \begin{proof}
        Replace heart $\cA_1$ with $\cA_2$, and proceed with the same argument as for the proof of Theorem \ref{thm.boundaryofW01}.
    \end{proof}
\end{thm}

\subsection{The case of $m=3$}
In this case, we observe $\sigma_{gl,3} = (Z_{gl,3},\cA_{gl,3})$ : 
$$
    Z_{gl,3} = (-1)((d_{\lambda_1}-k r_{\lambda_1} )\zeta_0 + (d_{\lambda_1} + (1-k) r_{\lambda_1})\zeta_1)  + (d_{\rho_2}-k' r_{\rho_2} )\zeta'_0 + (d_{\rho_2} + (1-k') r_{\rho_2})\zeta'_1
$$
$$
    \cA_{gl,3} = \cG l(p^* \cA(k)[1] \otimes \cO_{\hir}(-C_0), p^*\cA(k'))
$$
where $\zeta_j = x_j + iy_j \in \bH$ for $j=0,1$.

\begin{proposition} \label{cofficientsofpisigma3} 
     Let $\sigma = \sigma_{gl,3}$. Then the corresponding element $\pi(\sigma) \subset \rhom_\bZ(\Lambda,\bC)$ has component $(\xi_0,\xi_1, \xi_2, \xi_3)$ given by:   
    \begin{align*} 
        &\xi_0 = -(\zeta_0+\zeta_1+\zeta_0'+\zeta_1'), \\
        &\xi_{1} = \zeta_0'+\zeta_1', \\
        &\xi_{2} = -((\frac{1}{2}e+k)\zeta_0+(\frac{1}{2}e+k-1)\zeta_1)+ (\frac{1}{2}e-k')\zeta_0'+(\frac{1}{2}e-k'+1)\zeta_1' \\
        &\xi_3 = k \zeta'_0 + (k-1) \zeta'_1. 
    \end{align*}
 
    \begin{proof}
        We first expand $Z_{gl,3}$ (assuming $j_1=1, j_2=0$) in terms of $\mathrm{ch}(E)=(r,c_1,\mathrm{ch}_2)$ by substituting the formulas from Proposition \ref{Prop_rankdegofadjointfunctor} and \eqref{convDimvectToChern}. As in the cases m=1, 2, we solve the equation $Z_{gl,3}=\langle\pi(\sigma),\mathrm{ch}(E)\rangle_{M}$ for $\pi(\sigma)$.
        \begin{align*}
            Z_{gl,3} = & -((d_{\lambda_{1}}-k r_{\lambda_{1}}) \zeta_{0}+(d_{\lambda_{1}}+(1-k) r_{\lambda_{1}}) \zeta_{1}) \\
            & + ((d_{\rho_{2}}-k^{\prime} r_{\rho_{2}}) \zeta_{0}^{\prime}+(d_{\rho_{2}}+(1-k^{\prime}) r_{\rho_{2}})\zeta_{1}^{\prime})
        \end{align*}
        Rearranging the terms with respect to $r, c_1$, and $\mathrm{ch}_2$ yields:
        \begin{align*}
            Z_{gl,3} = & ((\zeta_{0}^{\prime}+\zeta_{1}^{\prime}) C_{0} + (-((\frac{1}{2} e+k) \zeta_{0}+(\frac{1}{2} e+k-1) \zeta_{1}) + (\frac{1}{2} e-k^{\prime}) \zeta_{0}^{\prime} +(\frac{1}{2} e-k^{\prime}+1) \zeta_{1}^{\prime})f) c_1 \\
            & - (k \zeta'_{0}+(k-1)\zeta'_{1})r \\
            & - (-(\zeta_{0}+\zeta_{1}) - (\zeta_{0}^{\prime}+\zeta_{1}^{\prime})) \mathrm{ch}_{2}
        \end{align*}
        We compare this to the Mukai pairing $\langle\pi(\sigma),\mathrm{ch}(E)\rangle{M} = c_1 \cdot (\xi_1 C_0 + \xi_2 f) - r\xi_3 - \mathrm{ch}2\xi_0$. By comparing the coefficients of $r, c_1$, and $\mathrm{ch}_2$ on both sides, we obtain the expressions:
        \begin{align*}
            \xi_{0} &= -(\zeta_{0}+\zeta_{1}+\zeta_{0}^{\prime}+\zeta_{1}^{\prime}), \\
            \xi_{1} &= \zeta_{0}^{\prime}+\zeta_{1}^{\prime}, \\
            \xi_{2} &= -((\tfrac{1}{2}e+k)\zeta_{0}+(\tfrac{1}{2}e+k-1)\zeta_{1})+(\tfrac{1}{2}e-k^{\prime})\zeta_{0}^{\prime}+(\tfrac{1}{2}e-k^{\prime}+1)\zeta_{1}^{\prime}, \\
            \xi_{3} &= k\zeta'_{0}+(k-1)\zeta'_{1}.
        \end{align*}
    \end{proof}
\end{proposition}

We denote $Z_{gl,3}|_{\cA_i} := Z_i$ for $i=1,2$ and denote $\cO(k) := p^*\cO_{\bP^1}(k)$ for short. Recall that by construction, we have $\zeta_{0} = Z_1(\cO(k-1)[1])$ and $\zeta_{1} = Z_1(\cO(k))$. Similarly, $\zeta'_{0} = Z_2(\cO(k-1)[1])$ and $\zeta'_{1} = Z_2(\cO(k))$. 
\begin{thm} \label{thm.m=3}
    The intersection $\partial \overline{S_{div}} \cap \cW_{0,3}$ is precisely $\cW_{0,3}$. In particular, $\cW_{0,3}$ is exactly contained in the boundary component $\partial_z$ \eqref{boundaryofwall}, and it intersects the vertex of $\partial \overline{S_{div}}$ if and only if  
    $$
        \rdet \begin{pmatrix}
             \rRe \zeta_0 & \rRe \zeta_1 \\
             \rIm \zeta_0 & \rIm \zeta_1
        \end{pmatrix} -         
        \rdet \begin{pmatrix}
             \rRe (\zeta_0+\zeta_1) & \rRe \zeta'_0 \\
             \rIm (\zeta_0+\zeta_1) & \rIm \zeta'_0
        \end{pmatrix}=0.
    $$
    \begin{proof}
        We will use the notation $\rdet(z_0,z_1) = \begin{pmatrix}
             \rRe z_0 & \rRe z_1 \\
             \rIm z_0 & \rIm z_1
        \end{pmatrix}$ for complex numbers $z_0,z_1$. Notice that, $\sigma \in \cW_{0,3}$ (namely, $\rper(\sigma)=0$) implies that, the phases of $Z_1(\lambda_1(\cO_x))$ and $Z_2(\rho_2(\cO_x))$ are equal. From Lemma \ref{phaseoflambda1Oxandrho2Ox}, this means $\phi(\zeta_0+\zeta_1) = \phi(\zeta_0' + \zeta_1')$. Therefore, there exists a positive real cofficient $c$ such that $\zeta_0'+\zeta_1' = c(\zeta_0+\zeta_1)$. Substituting this into the formulas from  Proposition \ref{cofficientsofpisigma3} yields the following.
        \begin{align*} 
            &\xi_0 = -(c+1)(\zeta_0+\zeta_1), \\
            &\xi_1 = c(\zeta_0+\zeta_1)
        \end{align*}
        Note that $z = t \cdot \rdet(\xi_{0,1})$ by \eqref{eq.z}. Thus the conditions $\cW_{0,3} \subset \partial_z$ is equivalent to the contision $\rdet(\xi_{0,1})=0$ holds for all $\sigma \in \cW_{0,3} $. From the expression above, $\xi_0$ and $\xi_1$ are proportional. This proportionality means the columns of the matrix for $\rdet(\xi_{0,1})$ are linearly dependent. Therefore $\rdet(\xi_{0,1})=0$.
        This obviously holds true regardless of the parameters (as long as $c>0$), and thus $\mathcal{W}_{0,3}\subset\partial_{z}$. From this, it follows that 
        $$
            \overline{\partial S_{div}}\cap\mathcal{W}_{0,3}=(\partial_{z}\cup\partial_{w})\cap\mathcal{W}_{0,3}=\mathcal{W}_{0,3}.
        $$
        The vertex $\partial_{z}\cap\partial_{w}$ corresponds to the locus within $\mathcal{W}_{0,3}$ satisfying both $z=0$ and $w=ze=0$. From \eqref{eq.w}, the condition $w=0$ is equivalent to $\rdet(\xi_{0,2})=0$. Substituting the above coefficients and rearranging, we have:
        \begin{align*}
            \rdet(\xi_{0,2}) &= -(c+1) \rdet(\zeta_0+\zeta_1,\xi_2) \\
            &= -(c+1) (\rdet(\zeta_0,\zeta_1) - \rdet(\zeta_0+\zeta_1,\zeta'_0)) \\
            &= 0
        \end{align*}
        Since $c > 0$ (as established earlier), we know $c+1 \ne 0$. This implies,  $
        \rdet(\zeta_0,\zeta_1) - \rdet(\zeta_0+\zeta_1,\zeta'_0) = 0$, which is the desired formula.  
        \end{proof}
\end{thm}

\begin{example}    
    A stability condition $\sigma_{gl,3} \in \mathcal{W}_{0,3}$ that intersects the vertex $\partial_{z}\cap\partial_{w}$ does indeed exist. For example, let the stability conditions on the components be $\sigma_1 = \sigma_{k,(-r+ri),(r+ri)}$ and $\sigma_2 = \sigma_{k',(r+ri),(-r+ri)}$ for some $r \in \mathbb{R}_{>0}$. If $\sigma_{gl,3} \in \mathcal{G}l(\sigma_1,\sigma_2)$ is the resulting glued stability condition, then it satisfies both the wall condition ($per(\sigma)=0$) and the vertex condition derived in Theorem \ref{thm.m=3}.
\end{example}

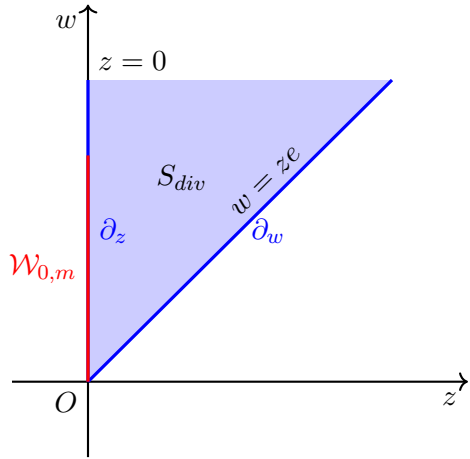
\begin{figure}[htbp]
  \centering 
  \begin{tikzpicture}
    \draw[->, thick] (-1, 0) -- (5, 0) node[below left] {$z$};
    \draw[->, thick] (0, -1) -- (0, 5) node[below left] {$w$};
    \node[below left] at (0,0) {$O$};

    \fill[blue!20] (0,0) -- (0,4) -- (4,4) -- cycle;
    \node[below left] at (1.7,3) {$S_{div}$};
    \draw[blue, very thick] (0,4) -- (0,0) node[midway,right] {$\partial_z$};
    \draw[blue, very thick] (0,0) -- (4,4) node[midway,right] {$\partial_w$};
    \draw[red, very thick] (0,3) -- (0,0) node[midway,left] {$\cW_{0,m}$};

    \node[above, rotate=45] at (2.5, 2.5) {$w=ze$};
    \node[above right] at (0, 4) {$z=0$};
  \end{tikzpicture}

  \caption{The region $\overline{S_{div}}$ for glued type $m=3$ at a fixed point $(x,y) \in \bR^2$. The wall $\cW_{0,3}$ intersects the boundary, and touches the vertex of $\overline{S_{div}}$.}
  \label{fig:zw_region3} 
\end{figure}

\section{Moduli spaces of $\sigma_{gl}$-stable objects}
Let $\sigma=\sigma_{gl,m}$ be a gluing stability condition of type $m=1,2,3,4$. We denote by $M^{\sigma ss}([\cO_x])$ the set of isomorphism classes of $\sigma$-semistable objects $F$ such that $\phi(F) = \phi(\cO_x)$ and $\rch(F) = \rch(\cO_x)$.  

Recall that if $\cA$ is a finite length abelian category, any object $E \in \cA$ admits a Jordan-H$\ddot{o}$lder filtration 
$$
    0 = E_0 \subset E_1 \subset ... \subset E_n = E
$$
such that every quotient $F_i := E_i / E_{i-1}$ is simple in $\cA$. In particular, the object
\begin{equation}
    {\rm gr}(E) = \bigoplus_{i=1}^n F_i
\end{equation}
is uniquely determined. Objects $E,E' \in \cA$ are called {\it S-equivalent} if ${\rm gr}(E) = {\rm gr}(E')$ ( \cite{HL} \S 1.5). We call $\rgr(E)$ the {\it semisimplification} of $E$. As shown in Proposition \ref{supportpropertyofgluingstabilitycondition}, any $\sigma_{gl,m}=(Z_{gl,m},\cA_{gl,m}) \in \cG l(\sigma_1,\sigma_2)$ is a locally finite, so the heart $\cA_{gl,m}$ has finite length.

\begin{lemma} \label{lemmasetofMsigma}
    Let $\sigma=\sigma_{gl,m} \in \cG l(\sigma_1,\sigma_2)$. We use the notation $\cO_{\Sigma_e}(n,m)$ for the line bundle $\cO_{\Sigma_e}(nC_0+mf)$. We also denote the following sets:
    \begin{align*}
        S_{p} := \{ \cO_x &\mid x \in \Sigma_e \}, \quad S_{f} := \{\cO_f \oplus \cO_f(-C_0)[1] \mid f \subset \Sigma_e \}, \\
        &S_{l} := \left \{\bigoplus_{j,l=0}^1 \cO_{\Sigma_e}(-l,k-j)[j+l] \right \}.
    \end{align*}
    Then, the set $M^{\sigma ss}([\cO_x])$ of $\sigma$-semistable objects with Chern character $\rch(\cO_x)$ is as follows:
    \begin{enumerate}
        \item If $m=1$, then
        $$
            M^{\sigma ss}([\cO_x]) = \begin{cases} 
                 S_p \cup S_f \qquad & {\rm if \ per(\sigma)=0 \ and \ per_i(\sigma) \neq 0} \\
                 S_p \cup S_f \cup S_l & {\rm if \ \rper(\sigma)=0 \ and \ \rper_i(\sigma) = 0} \\
                \emptyset  & {\rm if \ per}(\sigma) \neq 0
            \end{cases}
        $$
        \item If $m=2,3$, then
        $$
            M^{\sigma ss}([\cO_x]) = \begin{cases}
                S_p & {\rm if \ per}(\sigma) > 0 \\
                S_p \cup S_f \qquad & {\rm if \ per(\sigma)=0 \ and \ per_i(\sigma) \neq 0} \\
                S_p \cup S_f \cup S_l  & {\rm if \ per(\sigma)=0 \ and \ per_i(\sigma) = 0} \\
                \emptyset & {\rm if \ per}(\sigma) < 0
            \end{cases}
        $$
        \item If $m=4$, then $M^{\sigma ss}([\cO_x]) = S_p \cup S_f$
    \end{enumerate}
    \begin{proof}
        We first show the case of $m=1$. In this case, $\sigma_1$ is a quiver stability condition and $\sigma_2$ is a standard stability condition. Thus by Lemma \ref{phaseoflambda1Oxandrho2Ox}, we have $\phi(\cO_f)= 1$ and $ \phi(\cO_f(-C_0)[1]) \in (0,1]$. By definition \eqref{gluingperbersity}, we have $\rper(\sigma) = \phi(\cO_f(-C_0)[1]) - \phi(\cO_f) \leq 0$. If $\rper(\sigma) < 0$, then all skyscraper sheaves are unstable by Proposition \ref{PropHN-filtofAgl} (1). Namely, $M^{\sigma ss}([\cO_x]) = \emptyset$. Assume that $\rper(\sigma) = 0$. Then there are two subcases: $\rper_i(\sigma) = 0$ or $\rper_i(\sigma) \neq 0$. Note that in both subcases, $\cO_x$ is $\sigma$-semistable. If $\rper_i(\sigma) = 0$, then the semisimplification $\rgr(\cO_x)$ is given by
        \begin{equation} \label{semisimplificationofOx}            
            \rgr_0(\cO_x) = \cO_{\Sigma_e}(kf) \oplus \cO_{\Sigma_e}((k-1)f)[1] \oplus \cO_{\Sigma_e}(-C_0+kf)[1] \oplus \cO_{\Sigma_e}(-C_0+(k-1)f)[2]
        \end{equation}
        by Proposition \ref{PropHN-filtofAgl} (3). By a simple calculation, we have 
        \begin{align}
            &\rch_0(\rgr_0(\cO_x)) = 1-1-1+1 = 0, \\
            &\rch_1(\rgr_0(\cO_x)) = kf-(k-1)f-(-C_0+kf)+(-C_0+(k-1)f) = 0, \\
            &\rch_2(\rgr_0(\cO_x)) = \frac{2k+e}{2} + \frac{-2k+2-e}{2} = 1
        \end{align}
        and so we have $\rch(\rgr_0(\cO_x)) = \rch(\cO_x)$. If $\rper_i(\sigma) \neq 0$, the simplification is $\rgr(\cO_x) = \cO_f \oplus \cO_f(-C_0)[1]$. As shown in the calculation for Proposition \ref{PropHN-filtofAgl}, $\rch(\cO_f \oplus \cO_f(-C_0)[1]) = \rch(\cO_x)$. Since all objects listed $\cO_x, \cO_f \oplus \cO_f(-C_0)[1]$ and $\rgr_0(\cO_x)$ have the correct Chern character and $\sigma$-semistable when $\rper(\sigma)=0$, they belong to $M^{\sigma ss}([\cO_x])$. This justifies the statement for $m=1$. Next, we show the case of $m=2$. In this case, $\sigma_1$ is a standard and $\sigma_2$ is a quiver. We have $\phi(\cO_f(-C_0)[1])=1$ and $\phi(\cO_f) \in (0,1]$ by Lemma \ref{phaseoflambda1Oxandrho2Ox}. Thus, we have $\rper(\sigma)\geq 0$. Assume that $\rper(\sigma)>0$. Then the skyscraper sheaves $\cO_x$ are the only stable objects in $M^{\sigma ss}([\cO_x])$ by Proposition \ref{PropHN-filtofAgl} (2). The case of $\rper(\sigma)=0$ is the same as in the case of $m=1$. By combining the argument for $m=1$ and $m=2$, we obtain the argument for the case $m=3$, which corresponds to the situation where both $\phi(\cO_f),\phi(\cO_f(-C_0)[1]) \in (0,1]$. Finally, we show the case of $m=4$. In this case, $\sigma_1$ and $\sigma_2$ are both standard, so by Lemma \ref{phaseoflambda1Oxandrho2Ox}, $\phi(\cO_f(-C_0)[1]) = \phi(\cO_f)=1$, which implies this case only exists when $\rper(\sigma)=0$. Since $\rper_i(\sigma) = 1 \ne 0$, we are in the situation of Proposition \ref{PropHN-filtofAgl} (3)(a). This proves the assertion for $m=4$.  
        \end{proof}
\end{lemma}

\begin{lemma} \label{lemma.Sequivifandonlyifsamefiber}
     Let $\sigma=\sigma_{gl,m} \in \cG l(\sigma_1,\sigma_2)$ be a gluing stability condition satisfying ${\rm per}(\sigma) = 0$, and $p : \Sigma_e \ \to \bP^1$ be a projection. 
     \begin{enumerate}
         \item If $\rper_i(\sigma) \neq 0$ for $i=1,2$, then for any $x_1,x_2 \in \Sigma_e$, $p(x_1) = p(x_2)$ if and only if $\cO_{x_1}$ and $\cO_{x_2}$ are S-equivalent.
         \item If $\rper_i(\sigma) =0$ for $i=1,2$, then all skyscraper sheaves are S-equivalent.
     \end{enumerate}
     \begin{proof}
        We first prove (1). Assume $\rper_i(\sigma) \neq 0$ for $i=1,2$. By Proposition \ref{PropHN-filtofAgl} (3)(a), for any $x \in \Sigma_e$, the semisimplification $\rgr(\cO_x)$ is equal to $\cO_f \oplus \cO_f(-C_0)[1]$, where $f=p(x)$ is the fiber containing $x$. Thus, for any two points $x_1,x_2 \in \Sigma_e$,  $\rgr(\cO_{x_1}) = \rgr(\cO_{x_2})$ if and only if $x_1$ and $x_2$ belong to the same fiber. This is equivalent to $p(x_1)=p(x_2)$. We prove (2). Assume $\rper_i(\sigma) = 0$ for $i=1,2$. Then, by Proposition \ref{PropHN-filtofAgl} (3)(b) and \eqref{semisimplificationofOx}, the semisimplification $\rgr(\cO_x)$ is a specific object which does not depend on the choice of point $x$. Therefore, all skyscraper sheaves are S-equivalent.
    \end{proof}
\end{lemma}

To prove theorem \ref{Mainthm1}, we review the construction of algebraic spaces in \cite{Inaba02towardadefinitionofmoduliofcomplexesofcoherentsheavesonaprojectivescheme}. In \cite{Inaba02towardadefinitionofmoduliofcomplexesofcoherentsheavesonaprojectivescheme}, the functor from the category of locally noetherian schemes over fixed scheme $X$ over a scheme $S$ to the category of sets is defined as 
     $$
        Splcpx_{X/S}(T) = \left \{ E^\bullet \in D^b(X \times T) \  \middle | \   
            \begin{aligned}
                 & E^i \ is \ flat \ over \ T,\\ 
                 & \rExt_{X_t}^0(E_t,E_t) = \kappa(t) \ and \ \\ 
                 & \rExt_{X_t}^{-1}(E_t,E_t) = 0 \ for \ t \in T  \\
            \end{aligned}  
           \right  \},
     $$
and $Splcpx_{X/S}^{\acute{e}t}$ denotes the associated sheaf in the $\acute{e}$tale topology of $Splcpx_{X/S}$.

\begin{thm}[\cite{Inaba02towardadefinitionofmoduliofcomplexesofcoherentsheavesonaprojectivescheme} Theorem 0.2] \label{thm.Inaba.theorem0.2}
$Splcpx_{X/S}^{\acute{e}t}$ is represented by a locally separated algebraic space over $S$.    
\end{thm}

Let $\cM$ denote the algebraic space obtained in the above theorem. In \cite{toda2012stabilityconditionsextremalcontractions}, the author defined an open algebraic subspace $\cM^{\sigma}([\cO_x]) \subset \cM$ whose closed points are $\sigma$-stable objects in $M^{\sigma ss}([\cO_x])$. 

To prove Theorem \ref{Mainthm1}, we will use the same method as in \cite{toda2012stabilityconditionsextremalcontractions} Theorem 3.16. Although the stability conditions defined in \cite{toda2012stabilityconditionsextremalcontractions} differ from our construction, as is obvious from comparing Lemma \ref{lemmasetofMsigma} with Proposition 3.14 in \cite{toda2012stabilityconditionsextremalcontractions}, the properties of $\mathcal{M}^{\sigma}([\cO_x])$ are the same. Therefore, we can apply the method used in \cite{toda2012stabilityconditionsextremalcontractions} in the same way. Note that by definition, 
$$
    M^{\sigma ss}([\cO_x]) = \{ E \in \cP(\phi(\cO_x) \mid \rch(E) = \rch(\cO_x)\}.
$$

\begin{thm} \label{Mainthm1}
    Let $\sigma = \sigma_{gl,m}$ be a gluing stability condition of glued type $m$ (Definition \ref{typeofgluingstab}). Then, we have the following:

    \vspace{3mm}
    {\bf (Case $m=1$)}
    \begin{enumerate}
        \item If $\rper(\sigma)=0$ and $\rper_i(\sigma) \neq 0$ for $i=1,2$, then $\bP^1$ is the coarse moduli space of S-equivalence classes of objects in $M^{\sigma ss}([\cO_x])$.
        \item If $\rper(\sigma)=0$ and $\rper_i(\sigma)= 0$ for $i=1,2$, the moduli space of S-equivalence classes is isomorphic to $\rSpec \ \bC$.
        \item If ${\rm per}(\sigma) \neq 0$ then the moduli space of $\sigma$-semistable objects in $M^{\sigma ss}([\cO_x])$ is empty.
    \end{enumerate}

    \vspace{3mm}
    {\bf (Case $m=2,3)$}
    \begin{enumerate}
        \item If ${\rm per}(\sigma) > 0$, then $\hir$ is the fine moduli space of $\sigma$-stable objects in $M^{\sigma ss}([\cO_x])$.
        \item If $\rper(\sigma)=0$ and $\rper_i(\sigma) \neq 0$ for $i=1,2$, then $\bP^1$ is the coarse moduli space of S-equivalence classes of objects in $M^{\sigma ss}([\cO_x])$.
        \item If $\rper(\sigma)=0$ and $\rper_i(\sigma)= 0$ for $i=1,2$, the moduli space of S-equivalence classes is isomorphic to $\rSpec \ \bC$.
        \item If $\rper(\sigma) < 0$, then the moduli space of $\sigma$-semistable objects in $M^{\sigma ss}([\cO_x])$ is empty.
    \end{enumerate}

    \vspace{3mm}
    {\bf (Case $m=4$)}    
    $\\$ \hspace{5mm} In this case, $\bP^1$ is the coarse moduli space of S-equivalence classes of objects in $M^{\sigma ss}([\cO_x])$.
    
    \begin{proof}
     We first prove the case $m=2,3$. Assume $\sigma$ is a gluing stability condition of these types. If $\rper(\sigma) < 0$, then by Lemma \ref{lemmasetofMsigma}, $M^{\sigma ss}([\cO_x])$ is empty.

    Assume $\mathrm{per}(\sigma)>0$. By Lemma \ref{lemmasetofMsigma}, every $\sigma$-semistable
    object with $\mathrm{ch}=\mathrm{ch}(O_x)$ is $\sigma$-stable and is isomorphic to $\cO_x$ for a unique point $x\in \Sigma_e$. Hence we have a natural morphism
    \[
    \eta:\Sigma_e \longrightarrow \cM^\sigma([O_x]),\qquad x\mapsto [O_x],
    \]
    which is bijective on $\mathbb{C}$-valued points.
    
    We next show that $\eta$ is $\acute{e}$tale. Since $\cM^\sigma([O_x])$ is an open algebraic subspace
    of Inaba's moduli of simple complexes (Theorem \ref{thm.Inaba.theorem0.2}), its Zariski tangent space at $[\cO_x]$
    is canonically identified with $\rExt^1(\cO_x,\cO_x)$. On the other hand, $\rExt^1(\cO_x,\cO_x)\cong T_x\Sigma_e$. Therefore $d\eta_x$ is an isomorphism for all $x$, and $\eta$ is $\acute{e}$tale.
    
    Finally, $\eta$ is a monomorphism (two points cannot parametrize isomorphic skyscraper sheaves),
    and since $\Sigma_e$ is proper, $\eta$ is proper. A proper monomorphism of algebraic spaces is a closed immersion; combined with $\acute{e}$taleness it is also an open immersion. Hence $\eta$ is an open and closed immersion. Because it is bijective on points, $\eta$ is an isomorphism. This proves (1).

    Next, assume $\rper(\sigma) = 0$. Let $\underline{X}$ denote the functor $\rhom_{(Sch / \bC)}(-,X)$.
    If $\rper_i(\sigma) \neq 0$, Lemma \ref{lemma.Sequivifandonlyifsamefiber} (1) shows that two skyscraper sheaves $\cO_{x_1}$ and $\cO_{x_2}$ are S-equivalent if and only if $p(x_1) = p(x_2)$. 
    To prove (2), consider the functor 
    $$
        \mathcal{M} : ({\rm Sch}/\bC)^{\rm op} \to ({\rm Sets})
    $$
    which sends a $\bC$-scheme $T$ to the set of objects $\cE \in D^b(\hir \times  T)$ whose derived restriction $\cE_t := \cE|_{\hir \times \{t\}} $ belongs to $ \cP_{gl}(\phi(\cO_x))$ and satisfies $\rch(\cE_t) = \rch(\cO_x)$ for all $t \in T$. Let us consider the object
    $$
        \cF_T := \bfR (p \times \rid_T)_* \cE \in D^b(\bP^1 \times T).
    $$
    This object $\cF_T$ is a flat family of skyscraper sheaves of points in $\bP^1$ over $T$. Indeed, for each geometric point $t\in T$, Lemma \ref{lemmasetofMsigma} implies that $\cF_t$ is S-equivalent to $\cO_x$ for some $x \in \Sigma_e$.
    In the case $\rper(\sigma)=0$ and $\rper_i(\sigma)\neq 0$, Lemma \ref{lemma.Sequivifandonlyifsamefiber} (1) shows that the S-equivalence class is determined by $y := p(x)\in \bP^1$. Therefore $Rp_*\cF_t$ is isomorphic to the skyscraper sheaf $\cO_y$.
    Hence $\cF_T$ is a family of length-one $0$-dimensional sheaves on $\bP^1$ over $T$, which defines a morphism $T \to \bP^1$, 
    and therefore we have a natural transformation
    $$
        \Phi_{\bP^1} : \cM \to \underline{\bP^1}.
    $$

    Let $Z$ be another $\bC$-scheme and let $\Phi_Z : \mathcal{M} \to \underline{Z}$
    be any natural transformation. Applying $\Phi_Z$ to the universal family $\{\mathcal{O}_x\}_{x\in \Sigma_e}$,we obtain a morphism $g : \Sigma_e \to Z$. By Lemma \ref{lemma.Sequivifandonlyifsamefiber} (1), the $S$-equivalence class of $\cO_x$ depends only on the point $p(x) \in \bP^1$, hence $g$ is constant on each fiber of $p$. Let $s : \bP^1 \to \Sigma_e$ be a section of $p$ and define $h := g \circ s : \mathbb{P}^1\to Z$. Then for any $x \in \Sigma_e$ we have
    $$
        g(x)=g (s(p(x))) =h(p(x)),
    $$
    and therefore $g = h \circ p$. The morphism $h$ is unique since $p \circ s = \rid_{\bP^1}$.

    Moreover, by construction of $g$ from $\Phi_Z$, the equality $g=h\circ p$ implies that
    $\Phi_Z$ coincides with $\underline{h}\circ \Phi_{\mathbb{P}^1}$, where
    $\underline{h}:\underline{\mathbb{P}}^1\to \underline{Z}$ is the morphism of functors associated with
    $h:\mathbb{P}^1\to Z$, and $\Phi_{\mathbb{P}^1} : \cM\to \underline{\mathbb{P}}^1$ is the natural transformation
    constructed above (whose value on the universal family $\{\mathcal{O}_x\}_{x\in\Sigma_e}$ equals $p:\Sigma_e\to \mathbb{P}^1$).
    Hence every $\Phi_Z$ factors uniquely through $\underline{\mathbb{P}}^1$, and consequently
    $\mathbb{P}^1$ corepresents the functor $\cM$.

    If $\rper_i(\sigma) = 0$, Lemma \ref{lemma.Sequivifandonlyifsamefiber} (2) implies there is only a single S-equivalence class. Thus the moduli space is a point, $\rSpec \ \bC$.

    The assertion for $m=1$ follows from the same arguments, noting that $\rper(\sigma) \leq 0$. For $m=4$, we have $\rper(\sigma)=0$ and $\rper_i(\sigma) \neq 0$ (Lemma \ref{phaseoflambda1Oxandrho2Ox}), which corresponds to the case (2) above.
    \end{proof}
\end{thm}

\bibliographystyle{amsalpha}
\bibliography{reference}

\end{document}